\documentclass[11pt,a5paper]{article}

\usepackage{indentfirst}	
\usepackage[usenames]{color}
\usepackage[latin1]{inputenc}
\usepackage{bm}
\usepackage[english]{babel}
\usepackage{amsfonts, amsmath, amssymb,latexsym,amsthm, graphics,babel}
\usepackage[margin=2cm]{geometry}
\usepackage{geometry,graphicx}
\usepackage{enumerate}
\usepackage[all]{xy}
\usepackage{esint}
\usepackage{breqn}
\usepackage{hyperref}
\usepackage{comment}
\usepackage{mathrsfs}

\theoremstyle{plain}
\newtheorem*{theorem*}{Theorem}
\newtheorem{theorem}{Theorem}
\newtheorem{lemma}[theorem]{Lemma} 
\newtheorem{corollary}[theorem]{Corollary}
\newtheorem{proposition}[theorem]{Proposition}

\theoremstyle{definition}
\newtheorem{remark}[theorem]{Remark}

\newtheorem{definition}[theorem]{Definition}

\numberwithin{equation}{section} 
\newcommand{\R}{\mathbb R}

\newcommand{\C}{\mathbb C}
\newcommand{\Hyp}{\mathbb H}
\newcommand{\abs}[1]{\left\vert #1\right\vert}
\newcommand{\xs}{X\backslash S}

\newcounter{casenum}

\begin{document}

\title{Harmonic maps between 2-dimensional simplicial complexes 2}

\author{Brian Freidin\\Auburn University\\{\tt bgf0012@auburn.edu}
\and
Vict\`oria Gras Andreu\\
{\tt victoria.gras.andreu@gmail.com}}

\date{}

\maketitle


\begin{abstract}

We study metrics on two-dimensional simplicial complexes that are conformal either to flat Euclidean metrics or to the ideal hyperbolic metrics described by Charitos and Papadopoulos. Extending the results of our previous paper, we prove existence, uniqueness, and regularity results for harmonic maps between two such metrics on a complex.

\end{abstract}


\section{Introduction}

	This paper is a follow-up to \cite{FG}, and provides a step in the direction of producing Teichm\"uller maps between 2-dimensional simplicial complexes. The motivation for this work comes from \cite{CP}, where Charitos and Papadopoulos introduce a Teichm\"uller space of ideal hyperbolic metrics on a finite 2-dimensional simplicial complex. The Teichm\"uller space in \cite{CP} is parametrized in terms of classical shear coordinates, and compactified in terms of special measured foliations.
We proceed towards the goal of producing Teichm\"uller maps along the scheme proposed by Gerstenhaber and Rauch. In \cite{GR} they conjecture that Teichm\"uller maps between Riemann surfaces, which minimizes the complex dilataion over all maps in their isotopy class, are those that realize the min-max value
\[
\sup_{\rho}\inf_f E_\rho[f].
\]
Here the metrics $\rho$ range over the conformal class of the target surface, the maps $f$ range over a given isotopy class, and the quantity $E_\rho[f]$ is the Dirichlet energy of $f$ measured with respect to the metric $\rho$.
In \cite{K}, Kuwert proves that a Teichm\"uller map between Riemann surfaces is, in fact, harmonic with respect to a particular singular flat metric, thus producing the metric $\rho$ and the minimizing map $f$. Later, in \cite{Me}, Mese resolves the Gerstenhaber-Rauch conjecture by producing the Teichm\"uller map variationally according the min-max principle.
In \cite{FG}, the authors produced harmonic maps between ideal hyperbolic simplicial complexes. The present paper is concerned with producing harmonic maps between a broader class of metrics on a simplicial complex, thus approaching the full variational principle from \cite{GR}. The history of harmonic maps goes back at least as far as \cite{ES}, and harmonic maps between surfaces appear in \cite{J}, \cite{JS}, \cite{SY}. The theory of harmonic maps continues into the study of singular spaces with works including \cite{GS}, \cite{KS1}, \cite{J2}, \cite{EF}, \cite{DM1}, \cite{DM2}, \cite{DM3}.
Our main existence result concerns simplicial maps. A map between simplicial complexes is \emph{simplicial} if it sends vertices to vertices, edges to edges, and faces to faces. We consider the class $W^{1,2}(\rho)$ of simplicial maps $f:(X,\sigma)\to(X,\tau)$ between conformal classes on a complex $X$ with finite energy measured with respect to the metric $\rho$ in the conformal class of $\tau$, and the class $\mathcal{D}(\rho)$, of maps $u\in W^{1,2}(\rho)$ whose restriction to each open face of $X$ is a diffeomorphism. The latter class is motivated by the work of \cite{J}, \cite{JS}.
The class of simplicial maps is well-suited to our application for several reasons. First of all, the higher regularity results we state below rely pretty heavily on our simplicial assumption. Additionally, the classical Teichm\"uller maps turn out to be simplicial. This follows from \cite{K}, where Kuwert proves that the Teichm\"uller map pulls back a holomorphic quadratic differential on the target to one on the domain. To these holomorphic quadratic differentials one can associate a flat metric with cone-type singularities at the singular points of the differential, and find a triangulation of each surface with a set of vertices that contains the cone points. The structure of the Teichm\"uller map, as described in \cite{K}, then reduces to the piecewise affine map between these Euclidean triangulations.
In order to emulate the arguments of \cite{Me}, we must find minimizing mappings for target metrics that are not just Euclidean. Thus we establish the following  existance result.
\begin{theorem*}[c.f. Theorem~\ref{harmonic}]
	There exist energy minimizing mappings $u\in W^{1,2}(\rho)$ and $u_D\in\overline{\mathcal{D}}(\rho)$. That is,
	\[
	E(u) = \inf_{v\in W^{1,2}(\rho)}E(v),
	\]
	and
	\[
	E(u_D) = \inf_{v\in\mathcal{D}(\rho)}E(v).
	\]
\end{theorem*}
	
We also prove regularity of the above mentioned minimizing maps. In the case of surfaces, the regularity of harmonic maps follows from elliptic theory, see for instance the work of Morrey \cite{M1}, \cite{M2}. In the case of singular spaces, \cite{GS} and \cite{KS1} proves H\"older and Lipschitz continuity of maps with singular targets. When both domain and target are singular \cite{EF} proves H\"oler continuity. We prove the following results.
\begin{theorem*}[c.f. Theorem~\ref{global lip}]
	Let $u\in W^{1,2}(\rho)$ and $u_D\in\overline{\mathcal{D}}(\rho)$ be the energy minimizing maps constructed in Theorem~\ref{harmonic}.
	Then in any compact subset $V$ of $\xs$ there is a constant $C$ depending only on $V$ so that
	\[
	\abs{\nabla u}^2(p)\leq CE(u)\quad\text{ and }\quad\abs{\nabla u_D}^2(p)\leq CE(u_D)
	\]
	in $V$. As a result $u$ and $u_D$ are locally Lipschitz continuous, with Lipschitz constant at a point $p\in \xs$ depending only on the energy of $u$ (resp. $u_D$) and the distance of $p$ from the center of any face in which $p$ lies.
\end{theorem*}
\begin{theorem*}[c.f. Theorem~\ref{interior smooth}]
    Let $u\in W^{1,2}(\rho)$ and $u_D\in\overline{\mathcal{D}}(\rho)$ be the energy minimizing maps constructed in Theorem~\ref{harmonic}.  Suppose the metric $\rho^2\tau$ is $C^k$, with $3\leq k\leq\infty$. Then for any face $T$ of $X$, the restrictions $u\vert_T$ and $u_D\vert_T$ to the interior of $T$ are $C^{k-1}$ maps. If $\rho^2\tau$ is analytic, then $u\vert_T$ and $u_D\vert_T$ are analytic in the interior of each $T$. Moreover the map $u_D$ is a diffeomorphism on the interior of each face of $X$.
\end{theorem*}
\begin{theorem*}[c.f. Corollary~\ref{all the way to boundary regularity}]
Let $u:(X\backslash S,\sigma)\to(X\backslash S,\rho^2\tau)$ be a harmonic map. If $\rho^2\tau$ is $C^k$, then the restriction of $u$ to the closure of each face is $C^{k-1}$. If $\rho^2\tau$ is analytic, then the restriction of $u$ to the closure of each face is analytic.
\end{theorem*}
The paper is organized as follows. In Section~\ref{SBack} we introduce the spaces and classes of metrics with which the remainder of the paper is concerned. In Section~\ref{SFinite} we define our function spaces and produce finite energy maps $f\in \mathcal{D}(\rho)$ for any complex $X$ and metric $\rho$ introduced in the previous section.
In Section~\ref{SExist} we describe a harmonic replacement strategy and prove enough local regularity to carry out a minimization procedure and prove Theorem~\ref{harmonic}. In Section~\ref{SProp} we list properties of the minimizing maps, including the regularity discussed above and some topological properties.

In Section~\ref{SSing} we study degenerations of the metrics we introduced in Section~\ref{SBack} to metrics that are not smooth in each face of $X$. In particular, we allow cone-type singularities within the faces of $X$. We establish existence and Lipschitz regularity of harmonic maps into complexes with these degenerate metrics. In line with the observation that, indeed, classical Teichm\"uller maps between surfaces are simplicial maps with respect to suitably chosen triangulations, expanding the class of conformal metrics to include additional conical singularities makes the approach more flexible.
And in the last section we construct minimizing homogeneous approximations to the minimizing maps constructed earlier near the edges and vertices of the complex. The degree of homogeneity of the approximating map near $p$ is called the order of $u$ at $p$, and the main result of the last section establishes that for all of the minimizing maps we construct, the order does not vanish along the open edges of $X$.

\section{Background}\label{SBack}

\subsection{Simplicial Complexes}

\begin{definition}
	A \textit{2-dimensional simplicial complex} is a topological space $X$ together with two finite or infinite sets $\mathcal{C}$ and $\mathcal{F}$ that satisfy the following properties:
	\begin{enumerate}
		\item Each $T\in\mathcal{C}$ is a topological triangle. 
		That is, a topological space homeomorphic to a 2-dimensional closed disc with three distinguished points on the boundary called vertices. 
		The edges of $T$ are the closed segments of $\partial T$ bounded by two vertices and not containing the third.
		
		\item $\mathcal{F}$ is a maximal collection of homeomorphisms $f:A\to B$ where $A\subset T$ and $B\subset T'$ are distinct edges of (possibly identical) triangles $T,T'\in\mathcal{C}$ (along with the identity map on each edge of each triangle). For two edges $A,B$, there should be at most one map $f:A\to B$ in $\mathcal{F}$. The collection $\mathcal{F}$ is maximal with respect to two conditions. First, if $f:A\to B$ is in $\mathcal{F}$, then so is $f^{-1}:B\to A$. Second, if $f:A\to B$ and $g:B\to C$ are in $\mathcal{F}$, then so is $g\circ f:A\to C$.
		
		The elements of $\mathcal{F}$ are called gluing maps.
		
		\item As a topological space, $X$ is the quotient of the disjoint union $\coprod_{\mathcal{C}}T$ of the triangles in $\mathcal{C}$ by the equivalence relation identifying $x\in A$ with $f(x)\in B$ for each $f:A\to B\in\mathcal{F}$.
		
		Let $\pi:\coprod_{\mathcal{C}}T\to X$ be the quotient map. The fact that each $f\in\mathcal{F}$ is a homeomorphism and no two maps in $\mathcal{F}$ have the same domain and range implies that $\pi$ is injective on each edge of each triangle in $\mathcal{C}$. The image $\pi(T)$ of a triangle $T\in\mathcal{C}$ is called a \textit{face} of $X$, the image $\pi(e)$ of an edge $e$ of $T$ is called an \textit{edge} of $X$, and the image $\pi(v)$ of a vertex $v$ of $T$ is called a \textit{vertex} of $X$.
		
		We will also always impose an orientation on each edge of $X$. Pulling back by $\pi$, this puts an orientation on each edge $e$ of each triangle $T\in\mathcal{C}$ in such a way that the gluing maps $f\in\mathcal{F}$ are orientation-preserving homeomorphisms.

		\item $X$ is path connected.
		
		\item $X$ is locally finite. That is, each edge and each vertex is incident to finitely many faces.
	\end{enumerate}
\end{definition}

We list now some other notions regarding simplicial complexes.

\begin{remark}
	A simplicial complex is said to be \textit{(locally) 1-chainable} if the complement of the vertices, $\xs$, is (locally) connected. By the construction of our complexes, since faces are glued to other faces only along edges (and not at vertices alone), our complexes will satisfy this condition.
\end{remark}

\begin{definition}
	A 2-dimensional simplicial complex is \textit{finite} if the number of faces is finite. 
\end{definition}

We will always assume that our complexes are finite.

\begin{definition}
	Let $X$ be a 2-dimensional simplicial complex. 
	An edge $e$ of $X$ is called \textit{singular} if $e$ is incident to at least three faces. That is, the preimage of $e$ by the quotient map $\pi$ consists of three or more distinct edges.
\end{definition}

\begin{figure}[ht]
	\centering
	\includegraphics[width=0.5\textwidth]{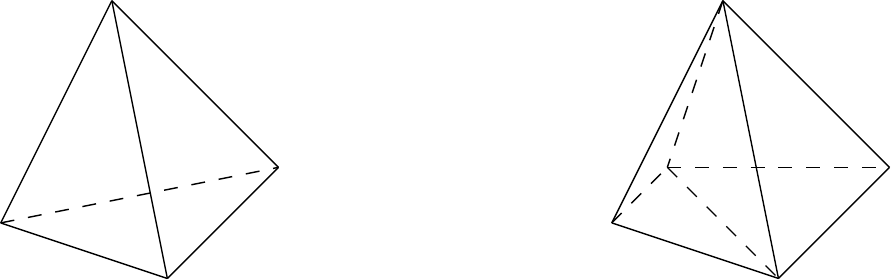}
	\caption{2-dimensional simplicial complexes with (right) and without (left) singular edges}
\end{figure}

\subsection{Metrics}\label{Smetrics}

We want to endow a 2-dimensional simplicial complex with a Riemannian metric. Since our objects are not smooth manifolds we will need to use a more general definition of a Riemannian metric for a complex. 

In what follows, we will assume that our complexes are finite.

\begin{definition}\label{RM}
	Let $X$ be a 2-dimensional simplicial complex with $\mathcal{C}=\{T_i\}_{i\in I}$ its set of triangles and $S$ its set of vertices.
	We'll say that $\sigma$ is a \textit{Riemannian metric} on $X$ (resp. $X\backslash S$) if
	\begin{itemize}
		\item the restriction of $\sigma$ to each closed face $\sigma\vert_{\bar{T}}$ (resp. $\sigma\vert_{\bar{T}\backslash S}$), is a Riemannian metric in the usual sense, and
		\item the restriction of $\sigma$ to each edge is well defined, i.e. if $T_1,\ldots,T_n$ meet at an edge $e$ and $\sigma_i=\sigma\vert_{\bar{T}_i}$ denotes the various restrictions, then $\sigma_i\vert_e=\sigma_1\vert_e$ for $i=2,\ldots,n$.
	\end{itemize}
\end{definition}

\begin{remark}
	Given a Riemannian metric $\sigma$ on $X$ (resp $\xs$), we induce a distance function on $X$ (resp. $\xs$) as follows. Given two points $x$ and $y$, any piecewise $C^1$ curve joining them has a length as measured by the Riemannian metric in each face of $X$. The distance between $x$ and $y$ is the infimum of the lengths of such curves.
\end{remark}

Two special classes of metrics we'll consider are Euclidean metrics on the compact simplicial complex and ideal hyperbolic metrics on the complex punctured at its vertices. We will also consider metrics conformal to these two special classes, which may or may not have additional conical singularities. We describe each of these different types of metrics below, along with their curvature properties. Towards this last goal, we give here a notion of curvature on singular spaces.

\begin{definition}
	Let $(Y,d)$ be a complete metric space. We say $(Y,d)$ is CAT($\kappa$), or that $(Y,d)$ has {\it curvature bounded from above by} $\kappa$,  if the following conditions are satisfied:
	\begin{enumerate}[(i)]
		\item $(Y, d)$ is a length space. 
		That is, for any two points $P, Q\in Y$, the distance $d(P,Q)$, which for simplicity we will also write as $d_{PQ}$, is realized as the length of a rectifiable curve $\gamma_{PQ}$ connecting $P$ to $Q$.
		We call such a curve a geodesic.
		\item Let $a=\sqrt{\abs{\kappa}}$. 
		Let $P, Q, R\in Y$ (assume $d_{PQ}+d_{QR}+d_{RP}<\frac{\pi}{\sqrt{\kappa}}$ for $\kappa>0$) with $Q_t$ defined to be the point on the geodesic $\gamma_{QR}$ satisfying $d_{QQ_t}=td_{QR}$ and $d_{Q_tR}=(1-t)d_{QR}$ Then we have
		\[
		\cosh(ad_{PQ_t})\leq \dfrac{\sinh((1-t)ad_{QR})}{\sinh(ad_{PQ})}\cosh(ad_{PQ})+\dfrac{\sinh(tad_{QR})}{\sinh(ad_{QR})}\cosh(ad_{PR})
		\]
		for $\kappa<0$,
		\[
		d^2_{PQ_t}\leq (1-t)d^2_{PQ}+td^2_{PR}-t(1-t)d^2_{QR}.
		\]
		for $\kappa=0$, and
		\[
		\cos(ad_{PQ_t})\geq \dfrac{\sin((1-t)ad_{QR})}{\sin(ad_{QR})}\cos(ad_{PQ})+\dfrac{\sin(tad_{QR})}{\sin(ad_{QR})}\cos(ad_{PR})
		\]
		for $\kappa>0$.
	\end{enumerate}
	
	If $(Y,d)$ does not necessarily satisfy the CAT($\kappa$) conditions, but every point $y\in Y$ has a neighborhood that does, we say $(Y,d)$ is {\it locally} CAT($\kappa$).
\end{definition}

If a space has curvature bounded above by $0$ we also refer to it as non-positively curved (NPC). The definition of CAT($\kappa$) given above, at least the local version, models the notion of sectional curvatures bounded above by $\kappa$ in a Riemannian manifold.

As a consequence of (ii), geodesics in an NPC space are unique, see for instance \cite{KS1}. 
It is also a well-known fact that if $Y$ is locally compact and NPC, then $Y$ is simply connected. 
Conversely, if $Y$ is a complete simply connected Riemannian manifold with non-positive sectional curvature, then $Y$ is an NPC space with the length metric induced from the Riemannian metric.

Furthermore, it is known that any CAT($\kappa$) space is also a CAT($\ell$) space for $\ell\geq\kappa$.

\subsubsection{Euclidean metrics}

To define Euclidean metrics on the compact simplicial complex $X$, we identify each triangle $T\in\mathcal{C}$ with a Euclidean triangle in such a way that we will induce a well-defined metric on each edge of $X$. Each edge of a Euclidean triangle is isometric to a segment in $\R$, which is characterized up to isometry by its length. So whenever $f:A\to B\in\mathcal{F}$ is a gluing map, $f$ is an isometry. That is, whenever two triangles are glued together along an edge, the corresponding edges must have the same length.

So at the very least we must specify a length function $\ell:E[X]\to\R_+$ from the set of edges of $X$ to the positive real numbers. And from elementary Euclidean geometry, a Euclidean triangle is uniquely determined (up to isometry) by a triple $(a,b,c)$ of edge lengths satisfying the triangle inequalities: $a<b+c$, $b<c+a$, and $c<a+b$. Once we have a length function satisfying these inequalities on each triangle, we can identify each triangle $T\in\mathcal{C}$ with a Euclidean triangle with the appropriate edge lengths, and use this identification to induce a metric on $T$. Moreover, these metrics on the triangles of $\mathcal{C}$ glue together and descend to a metric $\sigma$ on $X$ so that each face of $X$ is isometric to a Euclidean triangle.

Each Euclidean metric on $X$ is characterized by $E$ pieces of data, where $E=\#E[X]$ is the number of edges in the complex, so the space of such metrics can be identified with a subset of $\R_+^E$. Each triangle inequality restricts us to an open subset of $\R_+^E$. Since the length function $\ell\equiv 1$, identifying each face with an equilateral triangle, describes a metric satisfying all of the triangle inequalities, our space of Euclidean metrics is parametrized by a non-empty open subset of $\R_+^E$, and thus has dimension $E$.

It will often be convenient to work in coordinates. Given a Euclidean metric on $X$, and a triangle $T\in\mathcal{C}$, we can construct an isometry $\phi:\bar{T}\to\Delta$ to some triangle $\Delta\subset\R^2$. We use such a homeomorphism to induce coordinates in $\bar{T}$. If we have a chosen edge $A$ of $T$, we can arrange that $\Delta$ has one side along the $y$ axis and the remaining vertex in the right half-plane. In this case we introduce the notation $\Delta_{abc}$ to denote the Euclidean triangle with vertices $(0,0)$, $(0,a)$, and $(b,c)$ in $\R^2$.

\begin{definition}[Local model of an edge]\label{local edge euc}
Let $\sigma$ be a Euclidean metric on $X$. Let $e$ be an edge of $X$ with length $a$ and let $T_j$ enumerate the faces of $X$ incident to $e$. A system of isometries $\{\phi_j:T_j\to\Delta_{ab_jc_j}\}$ is called a {\it local model for $e$} if
\begin{enumerate}
    \item $b_j>0$ for all $j$ (hence $\phi_j(T_j)$ lies in the right half-plane for each $j$),
    \item $\phi_j(e)$ is the segment on the $y$-axis joining $(0,0)$ and $(0,a)$, and
    \item $\phi_j\vert_e=\phi_k\vert_e$ for all $j,k$.
\end{enumerate}
\end{definition}

In the last section we will also require coordinates suited to working near a vertex $v\in X$, for which we introduce the following model.

\begin{definition}[Local model of a vertex]\label{local vertex euc}
Let $\sigma$ be a Euclidean metric on $X$ and let $v\in X$ be a vertex. Enumerate the faces $T_j$ incident to $v$. A system of isometries $\{\phi_j:T_j\to\Delta_j\}$ is called a {\it local model for $v$} if
\begin{enumerate}
    \item Each $\Delta_j$ has a vertex at $(0,0)$, and $\phi_j(v)=(0,0)$ for each $j$
    \item Denoting by $e_j$ the instance of $e$ in the triangle $T_j$, if $f:e_j\to e_k\in\mathcal{F}$ then $\phi_j\vert_{T_j\cap T_k} = R\circ\phi_k\vert_{T_j\cap T_k}$, for some rotation $R$ about the origin.
\end{enumerate}
\end{definition}

\begin{proposition}
	Let $\sigma$ be one of the Euclidean metrics on $X$ described above. Then $(\xs,\sigma)$ is a locally CAT(0) space.
\end{proposition}

\begin{proof}
	Each face of $X$ with a Euclidean metric is CAT(0) since it is flat. For a point $p$ on an edge of $X$, a ball about $p$ that does not reach as far as any of the other edges or vertices is a union of flat half-discs glued along their common convex (geodesic) boundary pieces. By \cite{Sur}, Chapter 10, Corollary 5 and Lemma 9, such a set is CAT(0). Hence every point of $\xs$ has a CAT(0) neighborhood.
\end{proof}

\begin{remark}
	The Euclidean metrics we've defined may fail to be non-positively curved at the vertices of $X$. For example, a neighborhood of a vertex may be simply a cone, and if it is a cone of angle $<2\pi$ then it fails to have bounded curvature at the cone point. However, following \cite{Sur}, Chapter 10 Theorem 15, if associated to any loop around a vertex $v$ the corresponding faces have angles at $v$ that sum to at least $2\pi$, then the metric is CAT(0) in a neighborhood of $v$.
\end{remark}

\subsubsection{Hyperbolic metrics}

The ideal hyperbolic metrics we describe here are based on those defined in \cite{CP}. See also \cite{FG} for more details.

Consider the space $\xs$, and endow each face with the metric of the ideal hyperbolic triangle. To do so we take the upper half-plane model of $\Hyp^2$, that is $\Hyp^2=\{z\in\C\vert Im(z)>0\}=\{(x,y)\in\R^2\vert y>0\}$ with the Poincar\'e metric
\[
ds^2 = \frac{dx^2+dy^2}{y^2}.
\]
We call $\tilde{T}$ the ideal hyperbolic triangle that is the convex hull in $\Hyp^2$ of the points $0$, $1$, and $\infty$ in $\partial\Hyp^2$. For a triangle $T\in\mathcal{C}$ we choose a homeomorphism $\phi:T\backslash S\to\tilde{T}$ to induce the hyperbolic metric as well as the coordinates (either real $(x,y)$ coordinates or complex $z$ coordinates) in $T$.

Each edge of the ideal hyperbolic triangle is isometric to the real line $\R$, so for the metric to be well-defined on the edges, the gluing maps $f\in\mathcal{F}$ must be isometries of $\R$. The (orientation-preserving) isometries of $\R$ are all translations, parametrized by the translation distance. Thus ideal hyperbolic metrics on $\xs$ can be parametrized by the translation distances of the gluing maps, which are the same as the familiar shift parameters from the theory of punctured hyperbolic surfaces.

For any triangle $T\in\mathcal{C}$ and any edge $e$ incident to $T$, there is a preferred isometry $\phi:T\to\tilde{T}$ that maps the edge $e$ to the line joining $0$ and $\infty$ in $\Hyp^2$ in an orientation-preserving way (where the line joining $0$ and $\infty$ is oriented towards $\infty$). If we want to view all the faces incident to a common edge of $X$ at the same time, we use the following model.

\begin{definition}[Local model of an edge]\label{local edge hyp}
	Let $e\in X$ be an edge, and let $\{e_1,\ldots,e_n\}$ enumerate the preimages of $e$ under the quotient map $\pi$. Each edge $e_j$ is an edge of a triangle $T_j\in\mathcal{C}$. For each triangle $T_j$ choose $r_j>0$ and an isometry $\phi_j:T_j\to r_j\tilde{T}$ so that
	\begin{enumerate}
		\item $\phi_j$ maps $e_j$ onto the geodesic joining $0$ and $\infty$.
		\item For $f:e_j\to e_k\in\mathcal{F}$, $\phi_j\vert_{e_j} = \phi_k\circ f$.
		\item $\max_jr_j = 1$.
	\end{enumerate}
	Call the collection of isometries $\{\phi_j:T_j\to r_j\tilde{T}\}$ a \textit{local model for $e$}.
\end{definition}

\begin{remark}
	The maps $\phi_j$ constituting a local model for an edge depend on the metric of the complex $X$. Moreover, for a fixed metric, there are two possible local models for a given edge, corresponding to whether $e$ is identified with the geodesic joining $0$ and $\infty$ in an orientation-preserving or orientation-reversing way.
\end{remark}

We can also read the shift parameters from the local model of the edge. For two triangles $T_j$ and $T_k$ incident to $e$, the shift parameter gluing $T_j$ to $T_k$ along $e$ is simply $\log\left(\frac{r_j}{r_k}\right)$, where the numbers $r_j,r_k$ come from the relative scales of the corresponding triangles in the local model for $e$.

We can also see from the local model that if there are $k$ faces incident to a given edge $e$, then we have $k-1$ independent shift parameters for how to glue these faces along the given edge. Summing over all of the edges of $X$ we find a total of $3F-E$ parameters. However, we will also want to insist that our hyperbolic metrics on $\xs$ are complete. In analogy with the theory of punctured surfaces, the completeness of our metrics is equivalent to each punctured vertex having the geometry of a cusp.

To understand this condition we first introduce the link of a vertex. For a vertex $v\in X$, its link, denoted $lk(v)$, is a simplicial graph whose vertices correspond to edges of $X$ terminating at $v$ and whose edges correspond to faces of $X$ containing $v$ as a vertex. In the event that an edge $e$ has $v$ for both of its endpoints there will be two vertices in $lk(v)$ corresponding to $e$, or if there is a face $F$ that contains $v$ as more than one of its vertices there will be multiple edges in $lk(v)$ corresponding to $F$.

In order for the punctured vertex to have the geometry of the cusp, horocycles around the puncture must close up. This can be interpreted to say that the shift parameters around any closed loop in $lk(v)$ must sum to $0$. So the completeness of the metric imposes $rk(\pi_1(lk(v)))$ relations among the shift parameters around the vertex $v$, and we have similar relations around every other vertex. In total we impose $3F-2E+V$ relations. Subtracting from our $3F-E$ parameters from above, the space of complete ideal hyperbolic metrics on $\xs$ is parametrized by $\R^{E-V}$ and has dimension $E-V$. In the case where $\xs$ is homeomorphic to a punctured surface $S_{g,n}$, this dimension reduces to the familiar $6g-6+2n$, the dimension of the Teichm\"uller space.

For a complete hyperbolic metric on $\xs$, we also introduce preferred coordinates in a neighborhood of each punctured vertex. These coordinates are well defined precisely when the metric near the vertex is complete.

\begin{definition}[Local model of a vertex]\label{local vertex hyp}
	Let $v\in X$ be a vertex, and let $\{T_1,\ldots,T_n\}$ enumerate the faces of $X$ incident to $v$ (with repetition if $v$ occurs as the vertex of a face more than once). For each face $T_j$ choose $r_j>0$ and an isometry $\phi_j:T_j\to r_j\tilde{T}$ so that
	\begin{enumerate}
		\item $\phi_j$ maps the vertex $v$ to the point at $\infty$.
		\item For $f:e_j\to e_k\in\mathcal{F}$, $\phi_j\vert_{T_j\cap T_k} = \phi_k\vert_{T_j\cap T_k}+t$, for some real (horizontal) translation $t$.
		\item $\max_jr_j = 1$.
	\end{enumerate}
	Call the collection of isometries $\{\phi_j\}$ a \textit{local model for $v$}.
\end{definition}

In analogy to the Euclidean metrics introduced above, the ideal hyperbolic metrics have nice curvature properties. These results can also be found in \cite{CP} and \cite{FG}.

\begin{proposition}
	Let $\sigma$ be one of the complete ideal hyperbolic metrics on $\xs$ described above. Then $(\xs,\sigma)$ is a locally CAT(-1) space.
\end{proposition}

\begin{proof}
	The proof is identical to the proof for Euclidean metrics. Each face is isometric to a hyperbolic triangle, and so is CAT(-1) (in fact has curvature exactly -1). For any point $p$ on an edge, a ball about $p$ that does not reach as far as any other edge is a union of half-discs, each of curvature -1, glued along their common convex (geodesic) boundary pieces. By \cite{Sur}, Chapter 10, Corollary 5 and Lemma 9, such a set is CAT(-1). Hence every point of $\xs$ has a CAT(-1) neighborhood.
\end{proof}


\subsubsection{Conformal metrics}\label{conformal section}

We will also consider metrics conformal to those described above.

\begin{definition}\label{conformal}
	Let $\sigma$ be either a Euclidean metric on $X$ or an ideal hyperbolic metric on $\xs$, and let $\rho:X\to\mathbb{R}_{>0}$ be a continuous function that is $C^3$ on the closure of each face of $X$. We define the metric $\rho^2\sigma$ on each face by taking the Riemannian metric $\rho^2\sigma$, and by the continuity of $\rho$ these metrics are compatible on the edges of $X$. We also impose the following constraints on $\rho^2\sigma$:
	\begin{enumerate}
		\item\label{fix volume} The volume of $X$ is fixed. That is, if $d\mu_\sigma$ denotes the volume form of the metric $\sigma$, then
		\[ \sum_{T\in\mathcal{C}}\int_T\rho^2d\mu_\sigma = \sum_{T\in\mathcal{C}}\int_Td\mu_\sigma. \]
		\item\label{normal deriv} For every face $T$ and every edge $e$ incident to $T$, $\nabla(\rho\vert_T)\vert_e$ is perpendicular to $e$.
		\item\label{curv diffeq} If $K_\sigma$ is the curvature of the metric $\sigma$, then on each face
		\[ \Delta\rho \ge K_\sigma\rho + \frac{\abs{\nabla\rho}^2}{\rho}. \]
		\item\label{convex} Each face is convex with respect to the metric $\rho^2\sigma$.
		\item\label{complete} The complex is complete with respect to the metric $\rho^2\sigma$.
	\end{enumerate}
\end{definition}

If we started with a Euclidean metric $\sigma$ on $X$, the last condition of completeness is automatically satisfied. If we started with an ideal hyperbolic metric $\sigma$ on $\xs$, then it is possible that $\rho$ may blow up at the vertices of $X$, so long as it does not blow up fast enough to incur infinite area. We now investigate the second and third properties.

\begin{proposition}\label{geodesic}
	Given a metric $\sigma$, let $\rho$ be a conformal factor  as in Definition~\ref{conformal}. For each face $T$ and each edge $e$ incident to $T$, $e$ is a geodesic with respect to the metric $\rho^2\sigma\vert_T$ in $T$.
\end{proposition}

\begin{proof}
	For a fixed triangle $T\in\mathcal{C}$, choose an isometric embedding $\phi$ of $T$ into either $\R^2$ if $\sigma$ is Euclidean or $\Hyp^2$ if $\sigma$ is hyperbolic. By postcomposing with an isometry we may assume that $\phi$ takes one of the edges $e$ of $T$ onto the $y$-axis (using the upper half-space model for $\Hyp^2$). We use $\phi$ to push forward the metric $\rho^2\sigma$ from $T$ onto its image.
	
	Let $\gamma(t) = (0,f(t))$ be a parametrization of $\phi(e)$. In order for $\gamma$ to be a geodesic with respect to the metric $\rho^2\sigma$, we must have the geodesic equations, the first of which is
	\begin{eqnarray*}
		0 & = & \ddot{\gamma}^1 + ^\rho\Gamma_{ij}^1(\gamma)\dot{\gamma}^i\dot{\gamma}^j\\
		& = & ^\rho\Gamma_{22}^1(0,f(t))(f'(t))^2
	\end{eqnarray*}
	Here we have used $^\rho\Gamma$ to denote the Christoffel symbols of the metric $\rho^2\sigma$. If $\Gamma$ represents the Christoffel symbols of $\sigma$, then the conformal symbols satisfy
	\[
	^\rho\Gamma_{ij}^k = \Gamma_{ij}^k + \dfrac{1}{\rho}\left(\delta_i^k\partial_j\rho + \delta_j^k\partial_i\rho - \sigma_{ij}\sigma^{km}\delta_i^j\partial_k\rho\right).
	\]

	For the Euclidean metric all Christoffel symbols vanish, and for the Poincar\'e metric in the upper half-plane model, we have $\Gamma_{22}^1=0$. In both cases the metric $\sigma$ is diagonal. Hence the first of the geodesic equations for $\gamma$ reduces to
	\[
	0 = \sigma_{22}\sigma^{11}\partial_x\rho.
	\]
	
	From the condition~\ref{normal deriv} of Definition~\ref{conformal}, the normal derivative of $\rho$ along $e$ vanishes. In other words, $\partial_x\rho = 0$. Thus the first of the geodesic equations is automatically satisfied. The second of the geodesic equations reads
	\[
	f''(t) + ^\rho\Gamma_{22}^2(0,f(t))(f'(t))^2 = 0.
	\]
	Once the metric $\sigma$ and the conformal factor $\rho$ are fixed, this is an ordinary differential equation for the function $f$, which can be solved by classical theory. Therefore some parametrization of the edge $e$ is a geodesic with respect to the metric $\rho^2\sigma$ in the face $T$.
\end{proof}

\begin{proposition}\label{conformal curvature}
	Given a metric $\sigma$, let $\rho$ be a conformal factor as in Definition~\ref{conformal}. The curvature of the metric $\rho^2\sigma$ in each face $T$ is bounded above by $0$.
\end{proposition}

\begin{proof}
	Fix a face $T$ of $X$. Let $K_\sigma$ denote the Gaussian curvature of the metric $\sigma$ and $K_\rho$ the Gaussian curvature of the metric $\rho^2\sigma$. Let $\Delta$ denote the Laplacian with respect to the metric $\sigma$, which is given by the trace of the Hessian. Then the Gaussian curvatures satisfy
	\begin{eqnarray*}
		K_\rho & = & \frac{1}{\rho^2}\Big[K_\sigma - \Delta\log\rho\Big]\\
		& = & \frac{1}{\rho^2}\left[K_\sigma - \frac{\Delta\rho}{\rho} + \frac{\abs{\nabla\rho}^2}{\rho^2}\right]
	\end{eqnarray*}
	By the condition~\ref{curv diffeq} of Definition~\ref{conformal}, this curvature quantity is non-positive.
\end{proof}

\begin{corollary}\label{npcconf}
	The conformal metrics $\rho^2\sigma$, when $\rho$ is as in Definition~\ref{conformal}, define locally CAT(0) metrics on $\xs$.
\end{corollary}

\begin{proof}
	We have just shown that the metrics on each face are non-positively curved. For a point $p$ on an edge, a ball about $p$ that does not reach as far as any other edge or vertex is a union of NPC half-discs glued along their common convex (geodesic) boundary. By \cite{Sur}, Chapter 10, Corollary 5 and Lemma 9, such a set is NPC. Hence every point in $\xs$ has a CAT(0) neighborhood.
\end{proof}

\begin{remark}
	Just as in the case of Euclidean metrics, conformal metrics may fail to be CAT(0) at the vertices of the complex.
\end{remark}


\subsection{Singular metrics}

We will want to allow specific types of degenerations of our metrics. The particular degenerations we will consider are cone-type singularities of non-positive curvature, which we define in analogy to Section 4 of \cite{Tr}.

\begin{definition}\label{cone}
	Let $X$ be a finite 2-dimensional simplicial complex and let $\mathscr{C}=\cup_{n=1}^\infty\{p_n\}$ be a discrete set of points lying in the interior of the faces of $X$. A conformal metric $\rho^2\sigma$ , where $\rho$ may fail to be positive at $\mathscr{C}$, is called a \textit{cone metric} if it satisfies all five properties from Definition~\ref{conformal}, and in addition at each point $q\in\mathscr{C}$ there is an $\alpha=\alpha(q)>1$ and a positive $C^3$ function $\phi$ so that in polar coordinates centered at $q$ the conformal factor $\rho$ has the form
	    \[
	        \rho = r^{\alpha-1}\phi.
	    \]  
\end{definition}

Propositions~\ref{geodesic} and \ref{conformal curvature} still apply. For the curvature near a point $q\in\mathscr{C}$, the function $\phi$ from Definition~\ref{cone} satisfying $\rho = r^{\alpha-1}\phi$ must satisfy the same inequality as $\rho$, namely
\[
    \Delta\phi \geq K\phi + \frac{\abs{\nabla\phi}^2}{\phi}.
\]
Moreover one can see, as in \cite{K}, that the curvature picks up a Dirac measure at $q$ of mass $2\pi(1-\alpha)<0$.

Finally we show that metrics with cone-type singularities are well approximated by the smooth conformal metrics of Definition~\ref{conformal}.

\begin{lemma}\label{lemma1 kubert}
	Let $\rho^2\sigma$ be a metric with cone-type singularities as in Definition~\ref{cone}. There is a family $\{\rho_\epsilon\}_{\epsilon>0}$ of conformal factors satisfying Definition~\ref{conformal} so that $\rho_\epsilon\ge\rho_\tau$ for $\epsilon\ge\tau>0$ and so that $\rho_\epsilon^2\sigma$ converges to $\rho^2\sigma$ on the closure of each face as $\epsilon$ tends to $0$. This convergence is $C^0$ on $X$ and locally in $C^3$ on $T\backslash\mathscr{C}$ for each face $T$.
\end{lemma}

\begin{proof}
	This result is analogous to Lemma 1 from \cite{K}.
	
	By taking an appropriate mollification of $\rho$ near the points of $\mathscr{C}$ we find our family $\rho_\epsilon$ with the desired property. Concretely, fix a positive test function $\eta$ supported near the origin in $\R^2$. In coordinates near $q\in\mathscr{C}$, $\rho(x,y) = r^{\alpha-1}\phi(x,y)$, so near $q$ define
	\[
	\rho_\epsilon = \Big(\epsilon^3\eta(r/\epsilon) + r\Big)^{\alpha-1}\phi.
	\]
	The functions $\rho_\epsilon$ are smooth over $\mathscr{C}$, and if $\eta$ is chosen appropriately and $\epsilon$ is sufficiently small then the metrics have non-positive curvature. Furthermore it is clear that they satisfy the convergence properties claimed in the statement of the Lemma.
\end{proof}


\section{Existence of finite energy maps}\label{SFinite}

\subsection{Function spaces}

We define the classes of maps we'll be working on for the rest of this paper. 

\begin{definition}
	A map $u:X\to X$ {\it respects the simplicial structure} of $X$, or is a {\it simplicial map} if
	\begin{enumerate}
		\item $u\vert_T$ maps $T$ to $T$ for every $T\in\mathcal{C}$.
		\item $u\vert_e$ maps $e$ to $e$ for each edge $e$ of $X$.
		\item $u(v) = v$ for every vertex $v$ of $X$.
	\end{enumerate}
	A map $u:\xs\to\xs$ is a {\it simplicial map} if it satisfies the first two of these properties.
\end{definition}

\begin{remark}
    For a simplicial map $u:X\to X$, since the vertices are fixed by $u$, we may consider the map as a simplicial map $u\vert_{\xs}:\xs\to\xs$. 
\end{remark}

Throughout the rest of the paper, we will fix two metrics $\sigma$ and $\tau$, either both ideal hyperbolic metrics on $\xs$ or both Euclidean metrics on $X$ (which we may think of as metrics on $\xs$ for ease of notation). We will for the time being also fix a conformal factor $\rho$ as in Definition~\ref{conformal}, though in Section~\ref{SSing} we will allow $\rho$ to vary.

We will consider two main classes of maps for the remainder of this paper.  The first is
\[
W^{1,2}(\rho) = \{u:(\xs,\sigma)\to(\xs,\rho^2\tau)\mid u\text{ is simplicial, and }\forall T\in\mathcal{C}, u\vert_{T}\in W^{1,2}(T,T)\}.
\]
This class consists of all simplicial maps whose restriction to each face of $X$ is a $W^{1,2}$ map. The second class is
\[
\mathcal{D}(\rho) = \{u\in W^{1,2}(\rho) \mid \forall T\in\mathcal{C},u\vert_T\text{ is a diffeomorphism}\}.
\]
This class, motivated by the strategy of \cite{JS} for constructing harmonic diffeomorphisms, is the class of simplicial $W^{1,2}(\rho)$ maps whose restriction to each face is a diffeomorphism.

We will also require the spaces of functions
\[
L^2(\xs,\sigma) = \{f:\xs\to\R \mid \sum_{T\in\mathcal{C}}\int_T f^2d\mu_\sigma <\infty\},
\]
and
\[
W^{1,2}(\xs,\sigma) = \{f\in L^2(\xs,\sigma) \mid \forall T\in\mathcal{C}, f\vert_T\in W^{1,2}(T,\sigma)\}.
\]

\subsection{Existence}

The goal of this section is to produce a map in $W^{1,2}(\rho)$, i.e. a finite energy simplicial map from $(\xs,\sigma)$ to $(\xs,\rho^2\tau)$.

Recall that if $(X^m,g)$ and $(Y^n,h)$ are Riemannian manifolds, we define the energy of a smooth map $f:X\to Y$ as
\[
E(f):=\int_X\abs{\nabla f}^2d\mu.
\]
Here
\[
\abs{\nabla f}^2(x)=\sum_{i,j=1}^n\sum_{\alpha,\beta=1}^m g^{\alpha\beta}(x)h_{ij}(f(x))\dfrac{\partial f^i}{\partial x^{\alpha}}\dfrac{\partial f^j}{\partial x^{\beta}},
\]
with $x^{\alpha}$ and $f^i$ ocal coordinate systems around on $M$ and $N$, respectively.
$\abs{\nabla f}^2$ is called the energy density.

We can use the metrics $\sigma$ and $\tau$ to give coordinates on any face of $X$. In the case of hyperbolic metrics, we identify each face with $\tilde{T}\subset\Hyp^2$, the ideal triangle with vertices at $0$, $1$, and $\infty$. In the case of Euclidean metrics, we identify each face with a Euclidean triangle $\Delta\subset\R^2$. We may take a map $u:\xs\to\xs$, restrict to a face $T$, and use these coordinates to view $u\vert_T$ as a map $f:\tilde{T}\to\tilde{T}$ or $f:\Delta_1\to\Delta_2$. We can use the metrics $\sigma$ and $\rho^2\tau$ in these coordinates to compute the energy of $u\vert_T$. The total energy of $u$ is the sum over all the faces of $X$, namely

\[
E(u) = \sum_{T\in\mathcal{C}}E(u\vert_T).
\]

\begin{theorem}\label{finite energy euclidean}
Let $\sigma, \tau$ be Euclidean metrics. There exists a map $H\in\mathcal{D}(\rho)$. That is, $H:(\xs,\sigma)\to(\xs,\rho^2\tau)$ with finite energy such that $H\vert_T$ is a diffeomorphism for each face $T$ of $X$.	
\end{theorem}

\begin{proof}
First we'll define $H$ on each face of $X$. The definitions on adjacent faces will coincide on the common edge, so these local maps glue together to form a map $H\in\mathcal{D}(\rho)$.

Let $T\in\mathcal{C}$ be a face of $X$. From $\sigma$ we have an isometry $\phi_1:T\to\Delta_1$, where $\Delta_1\subset\R^2$ is a Euclidean triangle. And from the Euclidean metric $\tau$ we have an isometry $\phi_2:T\to\Delta_2$. If the vertices of $T$ are $v_1$, $v_2$, and $v_3$, then we can label the corresponding vertices of $\Delta_1$ by $v_{1,1}$, $v_{1,2}$, and $v_{1,3}$, and analogously the vertices of $\Delta_2$ by $v_{2,1}$, $v_{2,2}$, and $v_{2,3}$.

There is a unique affine map $f:\Delta_1\to\Delta_2$ that maps $v_{1,j}$ to $v_{2,j}$ for $j=1,2,3$. Define $h_T$ by
\[
h_T = \phi_2^{-1}\circ f\circ\phi_1.
\]
Since $h_T$ is affine on each edge of $T$, if $S$ and $T$ share an edge $e$ then $(h_S)\vert_e = (h_T)\vert_e$. Thus we can define $H:X\to X$ by
\[
H\vert_T = h_T.
\]
It is clear that $H$ is a simplicial map since each $h_T$ was. Also, as an affine map, $H\vert_T$ is a diffeomorphism for each face $T$.

Moreover the energy density of $H\vert_T$ with respect to the target metric $\tau$ is constant, and since $\rho$ is bounded the energy density of $H\vert_T$ is bounded with respect to the target metric $\rho^2\tau$. Since the area of $T$ with respect to the domain metric $\sigma$ is finite, the energy of $H\vert_T$ with respect to $\rho^2\tau$ is finite. And since there are finitely many faces, $H$ has finite energy. Hence $H\in\mathcal{D}$.
\end{proof}

\begin{theorem}\label{finite energy hyperbolic}
Let $\sigma, \tau$ be ideal hyperbolic metrics. There exists a map $H\in\mathcal{D}(\rho)$. That is, $H:(\xs,\sigma)\to(\xs,\rho^2\tau)$ with finite energy such that $H\vert_T$ is a diffeomorphism for each face $T$ of $X$.
\end{theorem}

\begin{proof}
We will construct a map as in the proof of Theorem 16 in \cite{FG}, by defining the map in a neighborhood of each vertex and filling in the compact remainder later on. Most of the proof is identical to the proof in \cite{FG}, so some of the details will be omitted.
	
Fix a vertex $v$ of $X$ and let $\{\phi_{j,\sigma}:T_j\to r_{j,\sigma}\tilde{T}\}$ be a local model for $v$ with respect to the metric $\sigma$, and similarly $\{\phi_{j,\tau}:T_j\to r_{j,\tau}\tilde{T}\}$ for the metric $\tau$. Consider the neighborhoods $N_*(v) = \{\phi_{j,*}^{-1}(z)\in r_{j,*}\tilde{T} \mid Im(z)>2\}$ for metrics $*=\sigma,\tau$, which are neighborhoods of the punctured vertices foliated by horocycles. Map $N_\sigma(v)$ to $N_\tau(v)$ as follows: for $x\in N_\sigma(v)\cap T_j$, define
\[
h_v(x) = \phi_{j,\tau}^{-1}\left(\frac{r_{j,\tau}}{r_{j,\sigma}}\phi^1_{j,\sigma}(x)+i \phi_{j,\sigma}^2(x)\right),
\]
where $\phi_{j,\sigma}=\phi_{j,\sigma}^1+i\phi_{j,\sigma}^2$.

By definition of the local models of vertices, the map $h_v$ is well defined on $N_\sigma(v)$. Note also that this map is a smooth diffeomorphism on the interior of each $N_\sigma(v)\cap T_j$ and injective on the intersection of each edge with $N_\sigma(v)$. In the coordinates defined by the local models, the map $h_v\vert_{T_j}$ is a horizontal stretch, whose energy density with respect to the hyperbolic metric $\tau$ is $(\frac{r_{j,\tau}}{r_{j,\sigma}})^2+1$.
Replacing the metric $\tau$ on the target with $\rho^2\tau$, the energy  density of $h_v\vert_{T_j}$ is $\rho^2((\frac{r_{j,\tau}}{r_{j,\sigma}})^2+1)$.
By condition~\ref{fix volume} in Definition~\ref{conformal}, the energy of $h_v$ on $N_\sigma(v)$ is finite, as there are finitely many faces incident to $v$ and each face has finite area with respect to $\rho^2\tau$.

Now define $h_w$ in $N_\sigma(w)$ for each vertex $w$ in $X$. The neighborhoods $N_\sigma(w)$ were chosen so that neighborhoods of different vertices do not intersect, so we may define $H = h_v$ in $N_\sigma(v)$ for each vertex $v$. Since there are finitely many vertices, we have used only finitely much energy in our construction thus far.

Let us now extend this map to the compact remainder of $X$. First fix an edge $e$ with endpoints $v$ and $w$. On $e\cap N_\sigma(v)$ let $H = h_v$ and on $e\cap N_\sigma(w)$ let $H = h_w$. On the remaining segment of $e$ define $H$ so that $H\vert_e$ is a smooth diffeomorphism. Do the same construction on each edge of $X$.

To finish the proof we need only to define $H$ in the compact remainder of the interiors of each face, namely in $T_j\backslash\cup_vN_\sigma(v)$. As $T_j\backslash\cup_vN_\sigma(v)$ is compact, any smooth diffeomorphism will have bounded energy density  and hence finite energy. One strategy to define $H$ in the compact remainder is as follows.

Let $X'\subset X$ be a compact subset of $X$ so that for each triangle $T_j$, $X'\cap T_j$ is convex with respect to the metric $\rho^2\tau$, with smooth boundary consisting of segments of edges of $T_j$ and curves within the neighborhoods $N_\tau(v)$ around the vertices of $T_j$. For an example of such a set, see the beginning of Section~\ref{SExist}. Let $\gamma$ be a curve in $X\cap T_j$ so that $H(\gamma) = \partial(X'\cap T_j)$. On the edges and on the neighborhoods $N_\sigma(v)$, $H$ is a smooth diffeomorphism, so $\gamma$ is a smooth curve bounding a compact region $\Omega_j\subset T_j$. Construct the solution $h_j:\Omega_j\to X'\cap T_j$ to the Dirichlet problem with boundary values given by the values of $H$ so far. That is, $h_j$ is a harmonic map with $h_j\vert_\gamma = H\vert_\gamma$. By Theorem 5.1.1 of \cite{J}, $h_j$ is a diffeomorphism. Define $H$ on each $\Omega_j$ (redefining on the intersection with the $N_\sigma(v)$ neighborhoods if necessary) by $H=h_j$.

The map constructed, following Theorem 16 in \cite{FG}, can be approximated by a finite energy map that is a diffeomorphism on each face, using \cite{IKO}, Theorem 1.2.

\end{proof}

\section{Existence of minimizing maps}\label{SExist}

In this section, we will prove the existence of maps from $(\xs, \sigma)$ to $(\xs, \rho^2\tau)$ which minimizes energy among the classes $W^{1,2}(\rho)$ and $\mathcal{D}(\rho)$ respectively, where $\sigma, \tau$ are either both Euclidean or both ideal hyperbolic metrics.

In order to achieve sufficient regularity in our minimization scheme, we will have to do a harmonic replacement argument. We will first describe our harmonic replacement scheme. In the case that we are minimizing in the class $\mathcal{D}(\rho)$ we will use Proposition 1 of \cite{JS}, so we will need some convex sets with smooth boundary in the target.

Whether we are working with conformally Euclidean metrics on the compact space $X$ or conformally hyperbolic metrics on the non-compact space $\xs$, we will think of our maps as maps $f:\xs\to\xs$ since we always preserve the vertices of the complex. Much of our construction, even in the case of conformally Euclidean metrics on $X$, break down near the vertices, so the perspective we take is needed in both cases.

We construct a compact exhaustion $X_n$ of $\xs$ consisting of sets analogous to $X'$ from the proof of Theorem~\ref{finite energy hyperbolic}. We require that for each face $T$ of $X$, $X_n\cap T$ is convex with respect to the metric $\rho^2\tau$ and has smooth boundary in $T$ consisting of segments of the edges of $T$ along with curves in the interior of $T$. We also require $X_n\subset X_{n+1}$ and $\cup_{n=1}^\infty X_n = \xs$. To construct such a set, cut off each vertex using geodesics in the incident faces. This leaves convex sets in each face, but the boundaries of the convex sets have corners. Simply approximate from within by a smooth convex set. To construct the whole sequence $X_n$, cut off a smaller neighborhood of each vertex for each successive set in our exhaustion.

Before we proceed we'll need some preliminary results.

\begin{lemma}\label{Wcomplete}
	Let $G\in W^{1,2}(\rho)$ be a given continuous map.
	Let $\{u_k\}_k\subset W^{1,2}(\rho)$  be an $L^2$ Cauchy sequence with $u_k\equiv G$ on $X\backslash X_n$ for each $k$.
	Suppose the energies of $u_k$ are uniformly bounded.
	Then the sequence converges in $L^2$ to a function $u\in W^{1,2}(\rho)$ with $u\equiv G$ on $X\backslash X_n$.
\end{lemma}

The proof of this result can be found in Lemma 17 of \cite{FG}. Even though it is stated for maps between ideal hyperbolic simplicial complexes, the same proof holds for a Euclidean metric $\tau$, or for a metric $\rho^2\tau$ conformal to either a Euclidean or ideal hyperbolic metric.

\begin{theorem}\label{poincare}
	\textbf{(Poincar\'e inequality for complexes).} 
	Let $K$ be a bounded connected compact subset of $X$ and $f\in W^{1,2}(\xs,\sigma)$ with compact support in $K$.
	Then there exists a constant $C$ depending only on $K$ and $X$ such that
	\[
	\int_K u^2d\mu_{\sigma}\leq C\int_K \abs{\nabla u}^2d\mu_{\sigma}.
	\]
\end{theorem}

The proof of this Theorem can be found in \cite{DM2}, Theorem 2.6.

\subsection{Harmonic Replacement}

Here we will construct harmonic maps on compact subsets of $\xs$ to use in our later arguments.

\begin{theorem}\label{dirichletW}
	Given $G\in W^{1,2}(\rho)$ and a set $X_n$ in our compact exhaustion of $X$, there exists $u\in W^{1,2}(\rho)$ minimizing energy among all maps $v\in W^{1,2}(\rho)$ with $v\equiv G$ on $X\backslash G^{-1}(X_n)$.

\end{theorem}

\begin{proof}
	
	The proof of this result also follows as in \cite{FG}, Theorem 19.

	Consider the class of maps $W^{1,2}_G$ of maps $v\in W^{1,2}(\rho)$ with $v\equiv G$ on $X\backslash G^{-1}(X_n)$. We suppress the dependence of $W^{1,2}_G(\rho)$ on the set $X_n$ in the notation.
	Let $\{v_k\}\subset W^{1,2}_G(\rho)$ be an energy minimizing sequence. 

	First replace $v_k$ with $u_k$, where $u_k = v_k = G$ on $X\backslash G^{-1}(X_n)$, and for each face $T$, $u_k$ solves the Dirichlet problem on $T\cap G^{-1}(X_n)$ with boundary data given by $v_k$. As $u_k$ solves a Dirichlet problem where it doesn't coincide with $v_k$, the energy of $u_k$ is no more than the energy of $v_k$, and $u_k$ also lies in $W^{1,2}_G(\rho)$ so it is also a minimizing sequence.	
	
	We will now show that $\{u_k\}$ is an $L^2$-Cauchy sequence. Define $u_{k,\ell}(x)$ to be the midpoint of the geodesic joining $u_k(x)$ and $u_\ell(x)$.
	Note that $u_{k,\ell}$ is well defined. 

	If $x$ lies in the interior of a face $T$, then by the convexity of the faces (condition \ref{convex} in Definition~\ref{conformal}), $u_k(x), u_\ell(x)\in T$ as well. Since by Corollary~\ref{npcconf} $X_n\cap T$ is NPC, there exists a unique geodesic joining $u_k(x)$ and $u_\ell(x)$ and therefore $u_{k,\ell}$ is well defined in the interior of each face. 
	If $x$ lies in an edge $e$, $u_k(x)$ and $u_\ell(x)$ also lie in that edge and therefore $u_{k,\ell}(x)\in e$ since edges are geodesics by Proposition~\ref{geodesic}.
	Finally, it is clear that $u_{k,\ell}=G$ on $G^{-1}(X_n)$ since both $u_k$ and $u_\ell$ do.
	It is also easy to check that $u_{k,\ell}\in L^2(\rho)$.

	From \cite{KS1} we have equation \eqref{bound}:
	\begin{equation}\label{bound}
	2E(u_{k,\ell})\leq E(u_k)+E(u_\ell)-\dfrac{1}{2}\int_{G^{-1}(X_n)}\abs{\nabla d_{\rho^2\tau}(u_k,u_\ell)}^2\leq 2K.
	\end{equation}	
	Since $\{u_k\}$ is a minimizing sequence, from \eqref{bound} we see that 	
	\[
	\lim_{k,\ell\to\infty}\int_{X_n}\abs{\nabla d_{\rho^2\tau}(u_k,u_\ell)}^2=0.
	\]
	Since $d_{\rho^2\tau}(u_k, u_\ell)\in W^{1,2}(\xs,\sigma)$ with compact support (Theorem 1.12.2 \cite{KS1}), we can use the Poincar\'e inequality of Theorem~\ref{poincare} and therefore the sequence $\{u_k\}$ is an $L^2$-Cauchy sequence. Morever the energies of $u_k$ are uniformly bounded by $E(u_k)\leq E(G)$.
	By Lemma~\ref{Wcomplete}, $\{u_k\}$ converges in $L^2(\rho)$ to a function $u\in W^{1,2}_G(\rho)$. 
	
	By semicontinuity of the energy, 
	\[
	E(u)\leq \liminf_k E(u_k)= \inf_{v\in W^{1,2}_G}E(v).
	\]
	Since $u\in W^{1,2}_G(\rho)$, we have $E(u) = \inf_{v\in W^{1,2}_G}E(v)$ as desired.

\end{proof}

We also have the same result for maps in the class $\overline{\mathcal{D}(\rho)}$.

\begin{theorem}\label{dirichletD}
	Given $G\in\mathcal{D}(\rho)$ and a compact set $X_n$ as in Theorem~\ref{dirichletW}, there exists $u\in \overline{\mathcal{D}(\rho)}$ minimizing energy among all maps in $v\in\overline{\mathcal{D}(\rho)}$ with $v\equiv G$ on $G^{-1}(X_n)$.
	Moreover, for each face $T$, the restriction of $u$ to the interior of $T\cap G^{-1}(X_n)$ is a harmonic diffeomorphism.
\end{theorem}

Given the proof of Theorem~\ref{dirichletW}, the proof of Theorem 20 in \cite{FG} holds in this case too.  

For $G\in\mathcal{D}(\rho)$ and a compact set $X_n$, denote by $\mathcal{D}_G(\rho)$ the class of maps $u\in\mathcal{D}(\rho)$ with $u\equiv G$ on $G^{-1}(X_n)$ (suppressing the dependence on $X_n$ in the notation).

\begin{proposition}\label{homotopic}	
	The energy minimizing maps in $W^{1,2}_G(\rho)$ and $\overline{\mathcal{D}}_G(\rho)$ constructed in the previous theorems are homotopic to $G$. 
\end{proposition}

This result follows from Proposition 21 in \cite{FG}.

\subsection{Local Regularity}\label{SLocReg}

Here we will establish enough regularity for the harmonic replacements constructed above so that we may use it in our global construction. 

For a point $p_0\in (\xs,\sigma)$, let $\text{st}(p_0)$ denote the union of all closed faces containing $p_0$. We will always assume that a ball centered at $p_0$ is contained in $\text{st}(p_0)$. There are two types of balls in $\xs$ that we will consider. In the case where $\sigma$ is a hyperbolic metric, we may identify each face $T$ with $\tilde{T}$ and consider the conformal Euclidean metric $dx^2+dy^2$ in $\tilde{T}$.
\begin{enumerate}[(i)]
    \item If $p_0$ lies in the interior of a face $T$, let $\phi:T\to\Delta$ (resp. $\phi:T\to\tilde{T}$) be an isometry if $\sigma$ is Euclidean (resp. hyperbolic). In case $\sigma$ is hyperbolic, take the conformal Euclidean metric $dx^2+dy^2$ in $\tilde{T}$. Let $r$ be less than $dist(\phi(p_0),\partial\phi(T))$ and let $B_r(p_0)$ denote $\phi^{-1}(B_r(\phi(p_0)))$, a topological disc.
    \item If $p_0$ lies on an edge $e$, let $\{\phi_j:T_j\to\Delta_j\}$ (resp. $\{\phi_j:T_j\to r_j\tilde{T}\}$) be a local model for the edge $e$ if $\sigma$ is Euclidean (resp. hyperbolic). In case $\sigma$ is hyperbolic, take the conformal Euclidean metric $dx^2+dy^2$ in each $r_j\tilde{T}$. Let $r$ be less than $dist(\phi_j(p_0),\partial\phi_j(T_j)\backslash\phi_j(e))$ for each $j$ and let $B_r(p_0)$ denote the union $\cup_j\phi_j^{-1}(B_r(\phi_j(p_0))\cap\phi_j(T_j))$. If there are $N$ faces incident to $e$, then $B_r(p_0)$ is homeomorphic to a union of $N$ half-discs glued along their common diameter.
\end{enumerate}

\begin{theorem}\label{3.9}
Let $p_0\in e$ be an edge point of $(\xs,\sigma)$, and let $u:B_r(p_0)\to(\xs,\rho^2\tau)$ be minimizing among maps that map $e\cap B_r(p_0)$ to $e$ and $T_j\cap B_r(p_0)$ to $T_j$ for each face $T_j$, and have the same trace on $\partial B_r(p_0)$. For a local model $\{\phi_{j,\alpha}:T_j\to\Delta_{j,\alpha}\}$ (resp. $\{\phi_{j,\alpha}:T_j\to r_{j,\alpha}\tilde{T}\}$) of $e$ with respect to the metrics $\alpha = \sigma,\tau$ in case $\sigma,\tau$ are Euclidean (resp. hyperbolic). Let $f_j = f_j^1+if_j^2 = \phi_{j,\tau}\circ u\circ\phi_{j,\sigma}^{-1}$ represent $u\vert_{T_j\cap B_r(p_0)}$ in coordinates. Define the Hopf differential
\[
\varphi_j = \left(\abs{\frac{\partial f_j}{\partial x}}^2 - \abs{\frac{\partial f_j}{\partial y}}^2 - 2i\left\langle\frac{\partial f_j}{\partial x},\frac{\partial f_j}{\partial y}\right\rangle\right)dz^2.
\]
Then $\varphi_j$ is holomorphic in the interior of $T_j\cap B_r(p_0)$, and moreover for all $y$ where defined,
\[
Im\sum_{j=1}^N\varphi_j(iy) = 0.
\]

\end{theorem}

The holomorphicity of the Hopf differential is a well-known fact for harmonic maps, and the proof of the Hopf balancing formula can be found in \cite{DM1}, Theorem 3.9. Our choice of conformal metrics guarantee that the proof still holds replacing $\tau$ with $\rho^2\tau$ regardless of whether the target metric is hyperbolic or Euclidean. 

\begin{theorem}\label{3.10}
	Let $p_0\in e$ be a point in an edge. If $u:B_{2r}(p_0)\to (\xs, \rho^2\tau)$ is minimizing in the same sense as in the previous Theorem, then with respect to the Euclidean metric in $B_{2r}(p_0)$, for each $j=1,\ldots,N$,
	\[
	\abs{\frac{\partial f_j}{\partial y}}^2(\bar{x}+i\bar{y})\leq \dfrac{2}{\pi r^2}E(u) \quad \text{and}\quad \abs{\frac{\partial f_j}{\partial x}}^2(\bar{x}+i\bar{y})\leq \dfrac{2N+2}{\pi r^2}E(u),
	\]
	where $(\bar{x},\bar{y})$ are the coordinates of a point in $B_r(p_0)$.
\end{theorem}

The proof of this theorem can also be found in \cite{DM1}, Theorem 3.10 and will still hold in our conformal class of metrics. 

In order to show local Lipschitz continuity we need to distinguish the cases where $\sigma, \tau$ are Euclidean or hyperbolic metrics. 

\begin{theorem}\label{local lip}
    Let $\sigma,\tau$ be either both Euclidean metrics on $X$ or both ideal hyperbolic metrics on $\xs$, and let $\rho^2\tau$ be a conformal factor as in Definition~\ref{conformal}. For a compact subset $X_n\subset(\xs,\rho^2\tau)$ as in Theorem~\ref{dirichletW} and a map $G\in W^{1,2}(\rho)$, let $u\in W^{1,2}_G(\rho)$ and $u_D\in \overline{\mathcal{D}(\rho)}$ be the minimizing maps produced by Theorem~\ref{dirichletW} and \ref{dirichletD}.
    For any $V$ compactly supported in the interior of $G^{-1}(X_n)$, there is a constant $C$ depending only on $V$ so that in $V$,
    \[
        \abs{\nabla u}^2 \leq CE(u) \quad \text{and} \quad \abs{\nabla u_D}^2 \leq CE(u_D).
    \]
    As a result, $u$ is locally Lipschitz continuous in $G^{-1}(X_n)$ with Lipschitz constant depending only on the energy of $u$ and the distance to $\partial G^{-1}(X_n)$.
\end{theorem}

\begin{proof}
    This proof is based on similar results in the literature. Because we have additional restrictions on our class of maps, namely that they respect the simplicial structure of $X$, we use an argument of Korevaar-Schoen in the interior of each face, and an argument of Daskalopoulos-Mese near the edges.
    
    Let $X_n$ be as in the statement, and let $r_0>0$ be sufficiently small so that for every point $q\in G^{-1}(X_n)$ on an edge, the ball $B_{r_0}(q)$ is defined as in the beginning of this section. In particular, for each such $q$ the ball $B_{r_0}(q)$ should be contained in the star of the edge in which $q$ lies. For $r<r_0$, let
    
    \[
        V_r = \{p\in G^{-1}(X_n) \vert dist(p,\partial G^{-1}(X_n))\geq r\}.
    \]
    Here the distance is measured with respect to the metric $\sigma$ in the domain.
    
    There are two types of points in $V_r$ to consider. First suppose that $p\in B_{r/2}(q)$ for some point $q\in G^{-1}(X_n)$ on an edge $e$ (see Figure~\ref{lip}). Since $u$ and $u_D$ are minimizing in $B_r(q)$ with respect to their boundary values, we may apply Theorem~\ref{3.10}. In the faces $T_j$ incident to $e$, the representations $f_j$ of the map in coordinates satisfy
	\[
	\abs{\frac{\partial f_j}{\partial y}}_{euc}^2\leq \dfrac{8}{\pi r^2}E(u) \quad \text{and}\quad \abs{\frac{\partial f_j}{\partial x}}_{euc}^2\leq \dfrac{8N+8}{\pi r^2}E(u),
	\]
	These estimates are taken with respect to the Euclidean metric in $B_r(q)$. In the case when $\sigma$ is Euclidean, this suffices to prove the result near the edges of $X$. In case $\sigma$ is hyperbolic, then we have
	\[
	\abs{\frac{\partial f_j}{\partial y}}_{hyp}^2 \leq \dfrac{8y^2}{\pi r^2}E(u) \quad \text{and}\quad \abs{\frac{\partial f_j}{\partial x}}_{hyp}^2\leq \dfrac{(8N+8)y^2}{\pi r^2}E(u).
	\]
	Since $V_r$ is compact, $y^2$ is bounded, so in either case we have bounded the partial derivatives of $u$ in a neighborhood of the edges.
	
    \begin{figure}[ht]
        \centering
        \includegraphics[width=\textwidth]{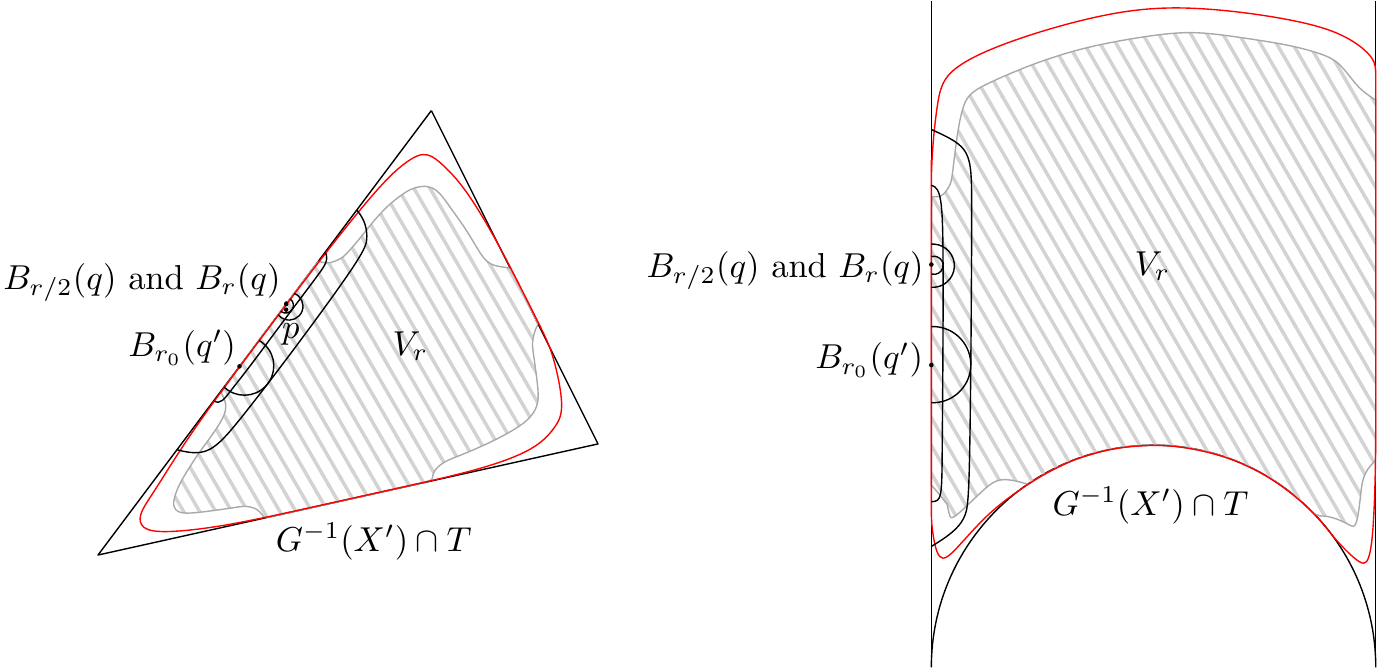}\label{lip}
        \caption{Local Lipschitz continuity on Euclidean and hyperbolic triangles}
    \end{figure}
	
	Now suppose that $p\in V_r$ is in the interior of some face $T$, but not in any $B_{r/2}(q)$ for any edge point $q$. Let $f$ represent either $u$ or $u_D$ in coordinates on the face $T$. Since $f$ is harmonic on $G^{-1}(X_n)$, Theorem 2.4.6 of \cite{KS1} implies that there is a constant C so that
	\[
	    \abs{\nabla f}^2(z) \leq \frac{C}{\min(1,dist(z,\partial G^{-1}(X_n))^2)}E(f).
    \]
    For $p\in V_r$, if $p$ is not contained in any $B_{r/2}(q)$ for any edge point $q$, then we have $dist(p,\partial G^{-1}(X_n))\geq r/2$. Hence at such a point,
    \[
        \abs{\nabla f}^2 \leq \frac{C}{r^2}E(f).
    \]
    This estimate on the energy density bounds the partial derivatives of $u$ in the interior of each face of $V_r$. Combined with the estimate near the edges, we have bounded the partial derivatives of both $u$ and $u_D$ in $V_r$, achieving the desired Lipschitz result.
\end{proof}

\subsection{Global Existence}

Now we may prove the global existence theorem.

\begin{theorem}\label{harmonic}
	There exist energy minimizing mappings $u\in W^{1,2}(\rho)$ and $u_D\in\overline{\mathcal{D}}(\rho)$. That is,
	\[
	E(u) = \inf_{v\in W^{1,2}(\rho)}E(v),
	\]
	and
	\[
	E(u_D) = \inf_{v\in\mathcal{D}(\rho)}E(v).
	\]
	Moreover both $u$ and $u_D$ are homotopic to the function $H$ constructed in Theorem~\ref{finite energy euclidean} for the Euclidean case or Theorem~\ref{finite energy hyperbolic} for the hyperbolic case.
\end{theorem}

	Using Theorems~\ref{dirichletW} and \ref{local lip}, the proof of this result follows exactly as in Theorem 26 of \cite{FG}.

\section{Properties of minimizing maps}\label{SProp}

In this section we will collect various properties of the minimizing map we constructed in Section~\ref{SExist}. First, we see that our map is locally Lipschitz. The same proof as for Theorem~\ref{local lip} proves the following:

\begin{theorem}\label{global lip}
	Let $u\in W^{1,2}(\rho)$ and $u_D\in\overline{\mathcal{D}}(\rho)$ be the energy minimizing maps constructed in Theorem~\ref{harmonic}.
	Then in any compact subset $V$ of $\xs$ there is a constant $C$ depending only on $V$ so that
	\[
	\abs{\nabla u}^2(p)\leq CE(u)\quad\text{ and }\quad\abs{\nabla u_D}^2(p)\leq CE(u_D)
	\]
	in $V$. As a result $u$ and $u_D$ are locally Lipschitz continuous, with Lipschitz constant at a point $p\in \xs$ depending only on the energy of $u$ (resp. $u_D$) and the distance of $p$ from the center of any face in which $p$ lies.
\end{theorem}

In fact, these maps are much more regular in the interior of each face. The regularity of the minimizing maps depends on the regularity of the metrics. Since both Euclidean and hyperbolic metrics are analytic, the only limiting factor on the regularity of the metrics comes from conformal factors.

\begin{definition}
    Let $\tau$ be either a Euclidean metric on $X$ or an ideal hyperbolic metric on $\xs$, and let $\rho^2\tau$ be a conformal metric as in Definition~\ref{conformal}. We say that the metric $\rho^2\tau$ is $C^k$ (resp. $C^\infty$, analytic) if the conformal factor is $C^k$ (resp. $C^\infty$, analytic) in the closure of each face $T$.
\end{definition}

\begin{theorem}\label{interior smooth}
    Let $u\in W^{1,2}(\rho)$ and $u_D\in\overline{\mathcal{D}}(\rho)$ be the energy minimizing maps constructed in Theorem~\ref{harmonic}.  Suppose the metric $\rho^2\tau$ is $C^k$, with $3\leq k\leq\infty$. Then for any face $T$ of $X$, the restrictions $u\vert_T$ and $u_D\vert_T$ to the interior of $T$ are $C^{k-1}$ maps. If $\rho^2\tau$ is analytic, then $u\vert_T$ and $u_D\vert_T$ are analytic in the interior of each $T$. Moreover the map $u_D$ is a diffeomorphism on the interior of each face of $X$.
\end{theorem}

\begin{proof}
    Fix a face $T$ and an isometry $\phi:T\to\Delta$ (if $\sigma,\tau$ are Euclidean) or $\phi:T\to\tilde{T}$ (if $\sigma,\tau$ are hyperbolic), and let $f = \phi\circ u\circ\phi^{-1}$, $f_D = \phi\circ u_D\circ\phi^{-1}$ represent $u, u_D$ in coordinates. For any disc $\mathbb{D}$ in the interior of $\phi(T)$, the maps $f$ and $f_D$ are both Lipschitz (c.f. Theorem~\ref{global lip}), so they have well defined boundary values on $\partial\mathbb{D}$. Since both $f$ and $f_D$ minimize energy in $\mathbb{D}$ with respect to their own boundary values, they must be harmonic in $\mathbb{D}$. By Theorem 3.2 of \cite{M1}, our harmonic maps are $C^2$ in the interior of each face.
    
    To improve the regularity of our maps, we can use Theorem 6.8.1 of \cite{M2}, which states that a $C^2$ map satisfying an elliptic partial differential equation on an open set is as regular as the coefficients of the equation. It is well known that the harmonic map equation is elliptic and the coefficients are given by the Christoffel symbols of the metric, that is, they depend on the first partial derivatives of the metric. Thus if $\rho^2\tau$ is $C^k$ then $f$ and $f_D$ are $C^{k-1}$ maps, and if $\rho^2\tau$ is analytic then $f$ and $f_D$ are analytic maps, all in the interior of the triangle.

    Finally, to show that $u_D$ is a diffeomorphism in the interior of each face we can follow the second half of the proof of Theorem 31 in \cite{FG} verbatim. 

\end{proof}

We now collect some topological properties satisfied by both $u$ and $u_D$, namely that these maps are proper and have degree 1.


\begin{theorem}\label{proper and degree}
The energy minimizing map $u\in W^{1,2}(\rho)$ and $u_D\in\mathcal{D}(\rho)$ constructed in Theorem~\ref{harmonic} are proper and have degree 1.
\end{theorem}

\begin{proof}
The proof of properness is trivial in the case when $\sigma, \tau$ are Euclidean metrics since continuous maps between compact spaces are proper. Moreover, $u$ and $u_D$ are homotopic (relative to the boundary of each triangle) to the affine map $H$ from Theorem~\ref{finite energy euclidean} and since affine maps have degree 1, so will $u$ and $u_D$.

In the case when $\sigma$ and $\tau$ are ideal hyperbolic metrics, the proof of this results follows word for word as in Theorems 34 and 35 of \cite{FG}.

\end{proof}

\subsection{Regularity of the $W^{1,2}(\rho)$ minimizer}

Here we collect properties of the minimizing map $u\in W^{1,2}(\rho)$. We begin with a balancing condition that characterizes the $W^{1,2}(\rho)$ minimizer.

\begin{definition}\label{balancing}
    Let $\sigma,\tau$ be either Euclidean metrics on $X$ or ideal hyperbolic metrics on $\xs$, as in Section~\ref{Smetrics}. Fix an edge $e$ of $X$, with incident faces enumerated $T_1,\ldots,T_n$, and let $\{\phi_j:T_j\to\Delta_{j,*}\}$ (in the case of Euclidean $\sigma,\tau$) or $\{\phi_j:T_j\to r_{j,*}\tilde{T}\}$ (in the case of hyperbolic $\sigma,\tau$) be local models for the edge $e$ with respect to the metrics $*=\sigma,\tau$ (recall that the edge $e$ is identified with an interval on the $y$-axis). Let $f_j = \phi_{j,\tau}\circ u\circ\phi_{j,\sigma}^{-1}$ represent $u\vert_{T_j}$ in coordinates, and write $f_j = f_j^1 + if_j^2$ in complex coordinates.
    
	We say $u$ satisfies the \textit{weak balancing condition} along $e$ if for any test function $\eta\in C^\infty_c(e)$ we have
	\[
	\sum_j\int_{\phi_{j,\sigma}(e)}\eta(\phi_{j,\sigma}(y))\frac{\partial f_j^2}{\partial x}(iy)dy = 0.
	\]
	We say that $u$ satisfies the \textit{strong balancing condition} along $e$ if
	\[
	\sum_j\frac{\partial f_j^2}{\partial x}(iy) = 0
	\]
	pointwise.
\end{definition}

\begin{proposition}
The minimizing map $u\in W^{1,2}(\rho)$ constructed in Theorem~\ref{harmonic} satisfies the weak balancing condition.
\end{proposition}

\begin{proof}
After changing notation, this is the content of Theorem 3 of \cite{DM3}. The result follows from the first variation formula for the energy functional.
\end{proof}

\begin{remark}
    In fact, a converse to the above proposition is true. Investigating the first variation formula for the energy functional shows that a map $u:(\xs,\sigma)\to(\xs,\rho^2\tau)$ is energy minimizing if and only if $u\vert_T$ is harmonic for each face $T$ and $u$ satisfies the weak balancing condition on each edge $e$.
\end{remark}

As a result of the balancing condition, the minimizing map $u\in W^{1,2}(\rho)$ enjoys additional boundary regularity.

\begin{proposition}\label{c1beta}
Fix an edge $e$ of $X$, and let $T_1,\ldots,T_n$ enumerate the faces incident to $e$. For the metrics $*=\sigma,\tau$, let $\{\phi_{j,*}:T_j\to\Delta_{j,*}\}$ (if $\sigma,\tau$ are Euclidean) or $\{\phi_{j,*}:T_j\to r_{j,*}\tilde{T}\}$ (if $\sigma,\tau$ are hyperbolic) be local models for $e$. Let $u\in W^{1,2}(\rho)$ be the minimizer from Theorem~\ref{harmonic}, and let $f_j = f_j^1 + if_j^2 = \phi_{j,\tau}\circ u\circ\phi_{j,\sigma}^{-1}$ represent $u\vert_{T_j}$ in coordinates.

Fix a point $iy\in\phi_{1,\sigma}(e)$. There is a neighborhood $\Omega$ of $iy$ and a constant $0<\beta<1$ so that $f_j^\alpha\in C^{1,\beta}(\Omega\cap\phi_{j,\sigma}(T_j))$ for each $j=1,\ldots,n$ and each $\alpha=1,2$.
\end{proposition}

\begin{proof}
The proof of this result is the same as the proof of Proposition 39 of \cite{FG}. We fix $r>0$ small enough so that $B_{j,r} = B_r(iy)\cap\{x\ge0\}\subset\phi_{j,\sigma}(T_j)$ for each $j$. Fixing an index $j_0$, we construct maps $\psi^\alpha$ extending $f^\alpha_{j_0}$ from the half disc $B_{r,j}$ to the disc $B_r(iy)$ and use elliptic regularity theory to show the result. We include the construction here since we'll need it in the future.

Fix an index $j_0\in\{1,\ldots,n\}$. For the function $f_{j_0}^2$, define
\[
\psi^2(x+iy) = \begin{cases}
f_{j_0}^2(x+iy), & x\ge0\\
-f_{j_0}^2(-x+iy) + \frac{2}{n}\sum_{j=1}^n f_j^2(-x+iy), & x<0.
\end{cases}
\]
For the function $f_{j_0}^1$, define
\[
\psi^1(x+iy) = \begin{cases}
f_{j_0}^1(x+iy), & x\ge0\\
-f_{j_0}^1(-x+iy), & x<0.
\end{cases}
\]

The Lipschitz continuity of $u$ from Theorem~\ref{global lip} implies that each $\psi^\alpha$ is Lipschitz continuous in $B_r(iy)$. Next we claim that for any test function $\varphi\in C^\infty_c(B_r(iy))$,
\[
\lim_{\epsilon\to0}\int_{B_r(iy)\cap \{x=\epsilon\}}\varphi\frac{\partial\psi^\alpha}{\partial x} = \lim_{\epsilon\to0}\int_{B_r(iy)\cap \{x=-\epsilon\}}\varphi\frac{\partial\psi^\alpha}{\partial x}.
\]
For $\alpha = 1$, use the fact that $f_{j_0}^1(it)=0$ for all $t$. For $\alpha = 2$, use the weak balancing condition of $u$.

From here the proof follows almost verbatim from \cite{DM3}. The functions $\psi^\alpha$ satisfy an elliptic equation with controlled coefficients, and the standard elliptic theory implies the regularity claimed.

\end{proof}

\begin{corollary}
The minimizing map $u\in W^{1,2}(\rho)$ from Theorem~\ref{harmonic} satisfies the strong balancing condition.
\end{corollary}

\begin{proof}
This follows immediately from the weak balancing condition and the regularity of Proposition~\ref{c1beta}.
\end{proof}

\begin{corollary}\label{all the way to boundary regularity}
Let $u:(X\backslash S,\sigma)\to(X\backslash S,\rho^2\tau)$ be a harmonic map. If $\rho^2\tau$ is $C^k$, then the restriction of $u$ to the closure of each face is $C^{k-1}$. If $\rho^2\tau$ is analytic, then the restriction of $u$ to the closure of each face is analytic.
\end{corollary}

\begin{proof}
Fist assume that $\rho^2\tau$ is $C^k$. In the interior of a face we know $u$ is $C^{k-1}$ from Theorem~\ref{interior smooth}. All that remains to check is that $u$ is $C^{k-1}$ up to the edges. This result follows from the bootstrapping argument of \cite{DM3}, Corollary 6, using the functions $\psi^\alpha$ defined in the proof of our Proposition~\ref{c1beta}.

If $\rho^2\tau$ is analytic, then the analyticity of the harmonic map follows from Theorem 6.8.2 of \cite{M2}. The argument is also detailed in \cite{FG}, Theorem 46
\end{proof}


\section{Singular metrics and energy minimizing maps}\label{SSing}

In this section we study the existence of energy minimizing maps from $(\xs, \sigma)$ to $(\xs, \rho^2\tau)$ where $\rho^2\tau$ is a cone metric as in Definition~\ref{cone}. We prove a more general result about compactness of harmonic maps, along the lines of Theorem 13 of \cite{Me}. The main differences between our results and \cite{Me} are that in our case our spaces are NPC, not necessarily compact and we require our sequence of metrics to be smooth.

\begin{theorem}\label{compactness}
Let $\sigma,\tau$ be either Euclidean metrics on $X$ or ideal hyperbolic metrics on $\xs$. Let $\{\rho_i^2\tau\}$ be a sequence of smooth conformal metrics on $\xs$ as in Definition~\ref{conformal}, and let $d_i$ denote the distance function associated to $\rho_i^2\tau$.
Let $H:(\xs, \sigma)\to (\xs, \tau)$ be a continuous, finite energy, simplicial map, e.g. from Theorem~\ref{finite energy euclidean} or Theorem~\ref{finite energy hyperbolic} (note that these maps have finite energy with respect to all of the $\rho_i^2\tau$ metrics), and let $u_i:(\xs, \sigma)\to (\xs, \rho^2_i\tau)$ be an energy minimizing simplicial map with respect to these metrics and in the homotopy class of $H$, from Theorem~\ref{harmonic}. Let $\delta_i = u_i^*d_i$ be the metric on the domain obtained by pulling back $d_i$ under $u_i$.

If the energy of $\{u_i\}$ is uniformly bounded and if $d_i$ converges locally uniformly to a distance function $d_0$, then there exists a subsequence $\{i'\}\subset\{i\}$ and a simplicial energy minimizing map $u_0$ with respect to $d_0$ so that $\delta_{i'}(\cdot, \cdot)$ converges uniformly to $d_0(u_0(\cdot), u_0(\cdot))$ and the energy of $u_{i'}$ converges to that of $u_0$.

Moreover the limit map $u_0$ is locally Lipschitz continuous.

\end{theorem}

\begin{proof}
Let $\{X_n\}$ be a compact exhaustion of the domain as described at the beginning of Section~\ref{SExist}. We first argue on the compact set $X_1$. Since the local Lipschitz bound of $u_i$ on $X_1$ depends only on the total energy of $u_i$ by Theorem~\ref{global lip}, we can take $L_1$ independent of $i$ so that 
\begin{equation}\label{bounded}
\delta_i(x,y)=d_i(u_i(x), u_i(y))\leq L_1d_\sigma(x,y),
\end{equation}
for $x,y\in X_1$.

Let $\mathcal{M}_1$ denote the family of metrics $\delta_i:X_1\times X_1\to \mathbb{R}$. It is easy to see that $\mathcal{M}_1$ is equicontinuous since for $(x_j, y_j)\in X_1\times X_1$, for $j=1,2$,
\begin{align*}
	\abs{\delta_i(x_1, y_1)-\delta_i(x_2, y_2)}&\leq \delta_i(x_1, x_2)+\delta_i(y_1, y_2)\\
	& \leq L_1(d_\sigma(x_1, x_2)+d_\sigma(y_1, y_2)).
\end{align*}
Note also that equation~\eqref{bounded} implies that $\mathcal{M}_1$ is uniformly bounded.

Thus, using the Arzel\`a-Ascoli Theorem, there exists a subsequence $\{\delta_i^1\}$ that converges uniformly on $X_1$. Repeating the same argument on $X_2$, there exists a subsequence of $\delta_i^1$, namely $\delta_i^2$ that converges uniformly on $X_2$. Using a diagonal method, we can take the sequence $\delta_i^i$ and this sequence will converge locally uniformly on $\xs$ to a metric $\delta_0:\xs\to\xs\to\mathbb{R}$. 

We now describe how to find the map $u_0$ with the property that
\[
\delta_0(\cdot, \cdot)=d_0(u_0(\cdot), u_0(\cdot)).
\]
Let $\mathcal{C}$ be a countable dense subset of $\xs$. Since each face is isometric to either an Euclidean or ideal hyperbolic triangle in $\mathbb{R}^2$, we can find a countable dense subset on each face, and since there are finitely many faces, the union of all these points is a choice of $\mathcal{C}$. 
 
We will need to establish pointwise boundedness of $\{u_i\}$. When $X$ is Euclidean, domain and target are compact and therefore, $\{u_i\}$ is uniformly bounded. If $\sigma, \tau$ are hyperbolic, we must show that for every $x_0\in\xs$ there exists $y_0\in \xs$ and $R>0$ independent of $\epsilon$ so that $d(u_i(x_0),y_0)\leq R$. Suppose for contradiction that for some $x_0\in\xs$ the sequence $\{u_i(x_0)\}$ is not pointwise bounded, so there is a subsequence $\{u_{i_\ell}(x_0)\}$ such that $u_{i_\ell}(x_0)$ escapes to a punctured vertex of $X$. The point $x_0$ lies in some face $T$, which we identify with the ideal hyperbolic triangle via an isometry $\phi:T\to\tilde{T}$. Without loss of generality we may assume that $\phi(u_{i_\ell}(x_0))$ escapes to the point at $\infty$.
	
Let $e$ be the edge opposite $\infty$, the hyperbolic geodesic joining $0$ and $1$. Fix $x_1\in \phi^{-1}(e)$, and take some compact set $X\subset\xs$ containing both $x$ and $x_1$. Since the sequence $u_{i_\ell}$ is uniformly Lipschitz continuous in $X_1$ we have
\begin{eqnarray*}
	d(\phi(u_{i_\ell}(x_0)),e) & = & \inf_{y\in e}d(\phi(u_{i_\ell}(x_0)),y)\\
	& \le & d(\phi(u_{i_\ell}(x_0)),\phi(u_{i_\ell}(x_1)))\\
	&\le & C d_\sigma(x_0,x_1).
\end{eqnarray*}
 Since $d(\phi(u_{i_\ell}(x_0)),e)$ is uniformly bounded in $\ell$, it cannot have escaped to $\infty$ as we supposed, a contradiction. Hence the sequence $\{u_i\}$ is pointwise bounded as desired.
	 
Therefore, for a fixed $x_1\in\mathcal{C}$, there exists a subsequence $\{u_i^1(x_1)\}\subset\{u_i(x_1)\}$ that converges. Then, for $x_2\in\mathcal{C}$, there exist a convergent subsequence $\{u_i^2(x_2)\}\subset\{u_i^1(x_2)\}$. Using a diagonal argument, we show $\{u_i^i(x)\}$ converges for all $x\in\mathcal{C}$. Set $u_0(x)=\lim_{i\to\infty}u_i^i(x)$ for $x\in S$. 

Let $x, y\in\mathcal{C}\cap X_1$. Then
\begin{align*}
d_0(p_x, p_y) = &\lim_{i\to\infty}d_0(u_i(x), u_i(y))\\
= & \lim_{i\to\infty}d_i(u_i(x), u_i(y))\\
\leq & L_1d_\sigma(x,y),
\end{align*}
where we've used the uniform convergence of the metrics in the second equality. A similar argument bounds $d_0(u_0(x),u_0(y))$ for $x,y\in\mathcal{C}\cap X_n$ for each $n$.
Thus, if a sequence $\{x_k\}\subset\mathcal{C}$ converges to $z\in\xs$ then for $k, \ell$ large enough, there is some $n$ so that
\[
d_0(u_0(x_k), u_0(x_\ell)) \leq L_nd_\sigma (x_k, x_\ell)<\epsilon.
\]
That is $\{u_0(x_k)\}$ is a Cauchy sequence. By completeness of the target, this sequence converges to some point $q\in\xs$. Define $u_0(z)=q$. Using this method we can define $u_0(z)$ for all points $z\in\xs$ since $\mathcal{C}$ was dense.

Finally, we need to see that this map is energy minimizing. Let $v\in W^{1,2}(\rho)$. Let $E_0(v)$ denote the energy of $v$ with respect to the distance $d_0$ and $E_i(v)$ the energy of $v$ with respect to $d_i$. Since $d_i$ converges uniformly to $d_0$, and by the dominated convergence theorem, for $i$ sufficiently large we have
\[
E_0(v)\geq E_i(v)-\epsilon\geq E_i(u_i)-\epsilon,
\]
where the last inequality follows from the assumption that $u_i$ minimize energy with respect to $d_i$. 
Moreover, since $\delta_i$ converges uniformly to $d_0(u_0(\cdot), u_0(\cdot))$, for $i$ large enough,
\[
E_i(u_i)\geq E_0(u_0)-\epsilon.
\]
Thus, by letting $\epsilon$ go to 0 we obtain that $E_0(u_0) \leq E_0(v)$ for all $v\in W^{1,2}(\rho)$ as we wanted. The Lipschitz continuity of $u_0$ is inherited from the sequence $u_i$.

\end{proof}

\begin{corollary}\label{harmonic-cone}
    For a cone metric $\rho^2\tau$, there exists an energy minimizing map $u\in W^{1,2}(\rho)$. That is,
	\[
	E(u)=\inf_{v\in W^{1,2}(\rho)} E(v).
	\]
	Moreover, $u$ is locally Lipschitz continuous.
\end{corollary}

\begin{proof}
    Let $\rho^2\tau$ be a cone metric as in Definition~\ref{cone}. By Lemma~\ref{lemma1 kubert} there are smooth conformal factors $\rho_\epsilon$ as in Definition~\ref{conformal} converging uniformly to $\rho$. Hence the induced metrics $d_\epsilon$ from the metrics $\rho_\epsilon^2\tau$ converge locally uniformly to the distance $d_0$ induced from $\rho^2\tau$.
    
    For each of the smooth conformal metrics $\rho_\epsilon^2\tau$ there is a harmonic map $u_\epsilon\in W^{1,2}(\rho_\epsilon)$ by Theorem~\ref{harmonic}. The Lipschitz bounds of Theorem~\ref{global lip} provide a uniform energy bound for the $u_\epsilon$.
    
    Hence, by Theorem~\ref{compactness} there is a sequence $\epsilon_i$ tending to $0$ so that the maps $u_{\epsilon_i}$ converge to a map $u_0$, and $u_0$ is energy minimizing with respect to the cone metric $\rho^2\tau$ and locally Lipschitz continuous.
\end{proof}

\section{Approximation by homogeneous maps}

This section will follow closely results of \cite{GS} and of \cite{DM1}. The maps $u$ and $u_D$ from Theorem~\ref{harmonic} have regularity on the interiors of each face as established in Theorem~\ref{interior smooth}, and the structure the harmonic maps of Corollary~\ref{harmonic-cone} around cone points was studied in \cite{K}. All that remains is to study the behavior along the edges and at the vertices of $X$. To that end, let $\sigma,\tau$ be two metrics on $X$, either both Euclidean or both ideal hyperbolic, $\rho$ a conformal factor as in Definition~\ref{conformal}, and $v:(\xs,\sigma)\to(\xs,\rho^2\tau)$ one of the minimizing maps from Theorem~\ref{harmonic} or Corollary~\ref{harmonic-cone}.

Let $p$ be either a point on some edge $e$ of $X$ or, if $\sigma$ and $\tau$ are Euclidean, a vertex of $X$, and let $q=v(p)$. If $p\in e$ is an edge point, enumerate the faces $T_j$ incident to $p$, consider local models $\{\phi_{j,\sigma}\}$ and $\{\phi_{j,\tau}\}$ for the edge $e$ (as in Definitions~\ref{local edge euc} and ~\ref{local edge hyp}), and shift the maps $\phi_j$ vertically until $\phi_{j,\sigma}(p)=\phi_{j,\tau}(q)=(0,0)$ for all $j$. If $\sigma$ and $\tau$ are Euclidean and $p=q=v\in X$ is a vertex, consider local models $\{\phi_{j,\sigma}\}$ and $\{\phi_{j,\tau}\}$ for $v$ as in Definition~\ref{local vertex euc}.

In either case let $r_0>0$ be such that the ball $B_{r_0}(p)$ (for both metrics $\sigma$ and $\tau$) lies entirely within the star $st(p)$ of $p$ (as discussed at the beginning of Section ~\ref{SLocReg}). For $r<r_0$ define the following functions.
\[
    E(r) = \int_{B_r(p)}\abs{\nabla v}^2d\mu_\sigma \qquad\text{and}\qquad I(r) = \int_{\partial B_r(p)}d^2(v,q)ds.
\]

\begin{lemma}\label{variations}
    For $r<r_0$,
    \[
        2E(r) \leq \int_{\partial B_r(p)}\frac{\partial}{\partial r}d^2(v,q)ds
    \]
    and
    \[
        E'(r) = \int_{\partial B_r(p)}\abs{\nabla v}^2d\mu_\sigma = 2\int_{\partial B_r(p)}\abs{\frac{\partial v}{\partial r}}^2ds.
    \]
\end{lemma}

\begin{proof}
The statements above are identical to ones found in the literature and follow from studying variations of the harmonic maps. Those variations considered continue to respect the simplicial structure of the maps, so the proofs carry forward to the present case.

For the first statement consider the contraction $Q_t:(B_{r_0}(q),\rho^2\tau)\to (B_{r_0}(q),\rho^2\tau)$ so that for each $y\in B_{r_0}(q)$ the curve $t\mapsto Q_t(x)$ is a constant speed geodesic from $q$ to $y$. By the convexity (item \ref{convex} of Definition~\ref{conformal}) and geodesic (Proposition~\ref{geodesic}) properties of the metric $\rho^2\tau$, the map $Q_t$ is a simplicial map. Now for a positive test function $\eta$ with support in $B_{r_0}(p)$ define
\[
	v_t(x) = Q_{t\eta(x)}(v(x)).
\]
Since $v=v_0$ minimizes energy,
\[
	\int_{B_{r_0}(p)}\abs{\nabla v}^2d\mu_\sigma \leq \int_{B_{r_0}(p)}\abs{\nabla v_t}^2d\mu_\sigma.
\]
Expanding as in \cite{EF} Lemma 10.2, or as in \cite{GS} Proposition 2.2, and letting $\eta$ approach the characteristic function of $B_r(p)$ yields
\[
	2E(r) \leq \int_{\partial B_r(p)}\frac{\partial}{\partial r}d^2(v,q)ds.
\]

For the second, consider the contraction $P_t:(B_{r_0}(p),\sigma)\to(B_{r_0}(p),\sigma)$ defined analogously to $Q_t$. In coordinates on each $T_j$ induced by the map $\phi_{j,\sigma}$, $P_t(x,y) = (tx,ty)$. For a positive test function $\eta$ with support in $B_{r_0}(p)$ define
\[
	v_t(x) = v(P_{t\eta(x)}(x)).
\]
Again the enery functional is minimized on $v=v_0$, so
\[
	\frac{d}{dt}\vert_{t=0}\int_{B_{r_0}(p)}\abs{\nabla v_t}^2d\mu_\sigma = 0.
\]
Expanding as in \cite{DM1} Proposition 3.2, or as in \cite{GS} leading up to equation (2.3), and again letting $\eta$ approach the characteristic function of $B_r(p)$ yields
\[
	E'(r) = 2\int_{\partial B_r(p)}\abs{\frac{\partial v}{\partial r}}^2ds.
\]
\end{proof}

\begin{lemma}\label{order}
    For $r<r_0$, the function
    \[
        r\mapsto \frac{rE(r)}{I(r)}
    \]
    is non-decreasing, and thus the limit
    \[
        \alpha = \lim_{r\to0}\frac{rE(r)}{I(r)}
    \]
    exists.
\end{lemma}

\begin{proof}
Again this fact exists in the literature, see e.g. \cite{GS} equations (2.4) and (2.5) or \cite{DM1} Proposition 3.2. Standard computation on each face $T_j$ gives the derivative of $I(r)$ as
\[
	I'(r) = \int_{\partial B_r(p)}\frac{\partial}{\partial r}d^2(v,q)ds + \frac{I(r)}{r}.
\]
Combining with the results of Lemma~\ref{variations}, the Cauchy-Schwarz inequality, and the triangle inequality (in the form of $\abs{\frac{\partial}{\partial r}d(v,q)}\leq\abs{\frac{\partial v}{\partial r}}$) yields
\[
	\frac{d}{dr}\log\left(\frac{rE(r)}{I(r)}\right) \geq 0.
\]
\end{proof}

A consequence of Lemma~\ref{order}, as expanded in \cite{DM1} Theorem 3.5, is the H\''older continuity of our minimizing maps. That result only establishes local H\''older continuity, but combined with the Lipschitz continuity of Theorem~\ref{global lip} and the finiteness of the complex $X$, we can upgrade the result.

\begin{corollary}
    When $\sigma$ and $\tau$ are Euclidean metrics, the minimizing maps $u\in W^{1,2}(\rho)$ and $u_D\in\overline{\mathcal{D}(\rho)}$ from Theorem~\ref{harmonic} are globally H\"older continuous.
\end{corollary}

\subsection{Homogeneous approximating maps}

To study the local behavior of $v$ near $p$ we will blow up the metrics $\sigma$ and $\rho^2\tau$ near $p$ and $q$ respectively. The following lemma will help us characterize the limiting map we find.

\begin{lemma}[GS Lemma 3.2]
    If $u_*:B_r(p)\to (X,\rho^2\tau)$ is harmonic and $\frac{rE(r)}{I(r)}=\alpha$ for all $0<r<r_0(p)$ then $u_*$ is intrinsically homogeneous of order $\alpha$.
\end{lemma}

The meaning of the phrase \emph{intrinsically homogeneous} of order $\alpha$ means that for $r<1$ and $x\in B_{r_0}(p)$ we have $d(u_*(rx),q)=r^\alpha d(u_*(x),q)$, and the curve $r\mapsto u_*(rx)$ traces the geodesic from $q$ to $u_*(x)$. The proof of the above lemma is identical to the one found in \cite{GS} Lemma 3.2, and involves studying the inequalities used in Lemma~\ref{order} in the case of equality.

In an effort to construct a homogeneous map, first represent the map $v$ in the coordinates in $T_j$ defined by the local models $\{\phi_{j,\sigma}\}$ and $\{\phi_{j,\tau}\}$ via the formula
\[
    f_j = \phi_{j,\tau}\circ v\circ \phi^{-1}_{j,\sigma}.
\]
Then for $\lambda,\mu>0$ define
\[
    f_{j,\lambda,\mu}(x,y) = \mu^{-1}f_j(\lambda x,\lambda y)
\]
and define $u_{\lambda,\mu}$ by
\[
    v_{\lambda,\mu} = \phi^{-1}_{j,\tau}\circ f_{j,\lambda,\mu}\circ\phi_{j,\sigma}\quad\text{in}\quad T_j.
\]

The map $v_{\lambda,\mu}$ should be seen as a map from $(\xs,\sigma)$ to $(\xs,\mu^{-2}\rho^2\tau)$. By a change of variables one can see that
\[
    \int_{B_r(p)}\abs{\nabla v_{\lambda\mu}}^2 = \lambda^2\mu^{-2}\int_{B_{\lambda r}(p)}\abs{\nabla v}^2 = \lambda^2\mu^{-2}E(\lambda r)
\]
and
\[
    \int_{\partial B_r(p)}d^2(v_{\lambda,\mu},q) = \lambda\mu^{-2}\int_{\partial B_{\lambda r}(p)}d^2(v,q) = \lambda\mu^{-2}I(\lambda r).
\]

Recall from Definition~\ref{conformal} that $\rho$ is $C^3$ on the closure of each face, including up to $p$. As $\mu\to0$ the metrics $\mu^{-2}\rho^2\tau$ on the neighborhoods $B_{\rho r_0}(q)$ converge to a flat metric $\tau^*$ on $st(q)$. If $\tau$ is Euclidean than the $\tau^*=\tau$ itself. If $\tau$ is hyperbolic, then $\tau^*$ is the flat metric conformal to $\tau$.
The following proposition produces a homogeneous map that approximates $v$, and its proof is identical to \cite{GS} Proposition 3.3 or \cite{DM1} Proposition 6.1.

\begin{proposition}
    Choosing $\mu = \big(\lambda^{-1}I(\lambda)\big)^{1/2}$, the maps $u_{\lambda,\mu}$ converge as $\lambda\to0$ to a non-constant homogeneous minimizing map $u_*:(B_r(p),\sigma)\to (st(p),\tau^*)$ of degree $\alpha>0$.
\end{proposition}

Unfortunately the structure near the punctured vertices of ideal hyperbolic complexes does not follow from the same computations of this section.

\subsection{Order at edge points}

In the coordinates given by the local models $\{\phi_{j,\sigma}\}$ and $\{\phi_{j,\tau}\}$, represent the homogeneous map $u_*$ constructed above in each $T_j$ by
\[
    f_{j} = \phi_{j,\tau}\circ u_*\circ\phi^{-1}_{j,\sigma}.
\]
The homogeneous maps $f_j$ all satisfy $f(tx,ty) = t^\alpha f(x,y)$.

\begin{theorem}
    If $p\in e$ is an edge point then the order $\alpha = \lim_{r\to0}\frac{rE(r)}{I(r)}$ of $u$ at $p$ is 1.
\end{theorem}

\begin{proof}
If $p$ is an edge point then the local models $\{\phi_{j,\sigma}\}$ and $\{\phi_{j,\tau}\}$ all have images lying in the right half-plane. As a result, the homogeneous maps $f_j$ described above all map the right half-plane to the right half-plane. Adopting polar coordinates $(r,\theta)$ in the domain, homogeneous harmonic maps $f$ of degree $\alpha$ all have the form
\[
    f(r,\theta) = \Big(r^\alpha(c_1\cos\alpha\theta+c_2\sin\alpha\theta),r^\alpha(c_3\cos\alpha\theta+c_4\sin\alpha\theta)\Big)
\]
for constants $c_1,c_2,c_3,c_4\in\mathbb{R}$.

Each $f_j$ maps the $y$-axis ($\theta = \pm\pi/2$) to the $y$-axis. Along $\theta=\pi/2$ this yields
\[
    c_1r^\alpha\cos(\alpha\pi/2) + c_2r^\alpha\sin(\alpha\pi/2) = 0.
\]
And along $\theta=-\pi/2$,
\[
    c_1r^\alpha\cos(\alpha\pi/2) - c_2r^\alpha\sin(\alpha\pi/2) = 0.
\]
If $\alpha$ is not an integer than this system has only the trivial solution $c_1=c_2=0$. But this would contradict the fact that $u_*$ is non-constant. Hence $\alpha\in\mathbb{Z}_{>0}$.

If $\alpha\geq 2$ then the first coordinate of $f_j$, $r^\alpha(c_1\cos\alpha\theta+c_2\sin\alpha\theta)$, changes signs in the interval $-\pi/2<\theta<\pi/2$. But the image of $f_j$ lies in the right half-plane, whose first coordinate is always positive. As a result, $\alpha\leq 1$. The only possibility left is $\alpha = 1$.
\end{proof}

\end{document}